\documentclass[11pt,twoside]{article}
\usepackage{amssymb,amsmath,amsthm,amsfonts,mathrsfs,hyperref}
\usepackage{times}
\usepackage{enumerate}
\usepackage{cite,titletoc}
\usepackage{color}
\usepackage[toc,page,title,titletoc,header]{appendix}
\usepackage{multirow}
\usepackage{graphicx}

\allowdisplaybreaks

\def\cA{\mathcal A}
\def\cB{\mathcal B}

\def\cF{\mathcal F}

\def\cK{\mathcal K}

\def\cO{\mathcal O}

\def\cW{\mathcal W}

\def\N{\mathop{\mathbb N\kern 0pt}\nolimits}
\def\Z{\mathop{\mathbb Z\kern 0pt}\nolimits}
\def\Q{\mathop{\mathbb Q\kern 0pt}\nolimits}
\def\R{\mathop{\mathbb R\kern 0pt}\nolimits}
\def\T{\mathop{\mathbb T\kern 0pt}\nolimits}
\def\SS{\mathop{\mathbb S\kern 0pt}\nolimits}

\def\ds{\displaystyle}
\def\f{\frac}
\def\al{\alpha}
\def\supp{\mathop{\rm supp}\nolimits}

\def\p{\partial}

\def\ve{\varepsilon}

\def\dive{\operatorname{div}}

\def\supp{\operatorname{supp}}
\def\ls{\lesssim}

\def\bn{\mathbf{n}}

\newcommand{\w}[1]{\langle {#1} \rangle}
\def\sgn{\operatorname{sgn}}

\hypersetup{colorlinks=true,linkcolor=blue,citecolor=red,urlcolor=cyan}

\theoremstyle{plain}
\newtheorem{theorem}{Theorem}[section]
\newtheorem{proposition}[theorem]{Proposition}
\newtheorem{lemma}[theorem]{Lemma}
\newtheorem{corollary}[theorem]{Corollary}

\newtheorem{remark}{Remark}[section]

\theoremstyle{definition}

\numberwithin{equation}{section}

\title{The sharp lifespan of small data smooth solutions to 2-D quadratic quasilinear wave equations in exterior domains}

\author{Bingbing Ding$^{1}$ \quad Fei Hou$^{2}$
    \quad Huicheng Yin$^{1,}$\footnote{Fei Hou (\texttt{fhou$@$nju.edu.cn}) and Huicheng Yin (\texttt{huicheng$@$nju.edu.cn}, \texttt{05407$@$njnu.edu.cn}) are supported by the National key research and development program of China (No.2024YFA1013301).
    In addition, Bingbing Ding (\texttt{13851929236$@$163.com}, \texttt{bbding$@$njnu.edu.cn}), Fei Hou and Huicheng Yin are supported by the NSFC (No.~12331007, No.~12571237).}\\
    [12pt]{\small 1. School of Mathematical Sciences and Mathematical Institute,}\\
    {\small Nanjing Normal University, Nanjing, 210023, China}\\
    {\small 2. School of Mathematics, Nanjing University, Nanjing, 210093, China}}

\begin{document}

\date{}
\maketitle
\thispagestyle{empty}

\begin{abstract}
In the paper [M. Keel, H. Smith, C.D. Sogge, Almost global existence for quasilinear wave equations in three space dimensions.
J. Amer. Math. Soc. 17 (2004), no. 1, 109-153], the authors prove that for the 3-D quadratic quasilinear wave equation in exterior domains
with homogenous Dirichlet boundary value and small initial data of size $\varepsilon$,
the lifespan ${\bar T}_{\varepsilon}$ of the smooth solution fulfills ${\bar T}_{\varepsilon}\ge e^{C/\varepsilon}$.
However, for the corresponding 2-D quadratic quasilinear wave equation in exterior domains
with homogenous Dirichlet or Neumann boundary value, so far it is still open whether
the expected sharp lifespan $T_{\varepsilon}\ge\frac{C}{\varepsilon^2}$ holds or not. In this paper,
we will solve this open question. Our main ingredients include: introducing the suitable Friedlander radiation field for the 2-D linear wave
equation in exterior domains with homogenous Dirichlet or  Neumann boundary value, constructing the delicate approximate solution,
and establishing some crucial space-time decay estimates for the solutions of 2-D quasilinear wave equation in exterior domains.
On the other hand, for the radial symmetric solutions to a class of 2-D quadratic quasilinear wave equation in exterior domains, the upper bound of the lifespan $T_{\varepsilon}\le\frac{C}{\varepsilon^2}$ is derived and the sharp constant $C$ is also determined
explicitly.
\vskip 0.2 true cm

\noindent
\textbf{Keywords.} Lifespan, quadratic quasilinear wave equation, Friedlander radiation filed,

\qquad \quad $L^{\infty}-L^{\infty}$ estimate, approximate solution, weighted energy estimate

\vskip 0.2 true cm
\noindent
\textbf{2020 Mathematical Subject Classification.}  35L05, 35L20, 35L72
\end{abstract}

\vskip 0.2 true cm

\addtocontents{toc}{\protect\thispagestyle{empty}}
\tableofcontents

\section{Introduction}

\subsection{Main results}
In this paper, we focus on the following IBVP (initial boundary value problem) of the 2-D quadratic quasilinear wave
equation in exterior domains with homogeneous Dirichlet or Neumann boundary condition
\begin{equation}\label{QWE}
\left\{
\begin{aligned}
&\Box u+Q(\p u,\p^2u)=0,\hspace{2.8cm} (t,x)\in[0,+\infty)\times\cK,\\
&\cB u(t,x)=0,\hspace{4.3cm} (t,x)\in[0,+\infty)\times\p\cK,\\
&(u,\p_tu)(0,x)=(u_0,u_1)(\ve,x),\qquad\qquad  x\in\cK,
\end{aligned}
\right.
\end{equation}
where $\p=(\p_0,\p_1,\p_2)=(\p_t,\p_{x_1},\p_{x_2})$, $x=(x_1,x_2)$, $\ds\Box=\p_t^2-\Delta$ with $\Delta=\p_1^2+\p_2^2$,
$\cK=\R^2\setminus\cO$,
the obstacle $\cO\subset\R^2$ is compactly convex and contains the origin, the boundary $\p\cK=\p\cO$ is smooth,
$(u_0,u_1)(\ve,x)\in C^\infty([0,1)\times\cK)$, $u_k(0,x)=0$ $(k=0,1)$ and $\supp_x(u_0,u_1)\subset\{x\in\cK:|x|\le M_0\}$
with some fixed constant $M_0>1$.
In addition, $\cB={\rm Id}$ or $\cB=\frac{\p}{\p\bn}$ represents the Dirichlet or Neumann boundary condition, respectively,
where $\bn=(n_0,n_1,n_2)=(0, n_1(x),n_2(x))$ is the unit outer normal of $[0,+\infty)\times\p\cK$.
The nonlinearity $Q(\p u,\p^2u)$ is given by
\begin{equation}\label{nonlinear}
Q(\p u,\p^2u)=\sum_{\alpha,\beta,\mu=0}^2Q^{\alpha\beta\mu}\p^2_{\alpha\beta}u\p_{\mu}u,
\end{equation}
where $Q^{\alpha\beta\mu}$ are constants and $Q^{\alpha\beta\mu}=Q^{\beta\alpha\mu}$ hold.
It is assumed that the first null condition fails for problem \eqref{QWE}, namely,
\begin{equation}\label{1st:fail}
Q_{1st}(\tilde\omega)=\sum_{\alpha,\beta,\mu=0}^2Q^{\alpha\beta\mu}\omega_\alpha\omega_\beta\omega_\mu\not\equiv0,
\end{equation}
where $\tilde\omega=(\omega_0,\omega_1,\omega_2)=(-1,\cos\theta,\sin\theta)$, $\omega=(\omega_1, \omega_2)=(\cos\theta,\sin\theta)$
and $x=r\omega$ with $r=|x|=\sqrt{x_1^2+x_2^2}$ and $\theta\in [0, 2\pi]$.

The compatibility condition of $(u_0,u_1)$ is necessary for finding the smooth solutions of \eqref{QWE}.
Set $J_ku=\{\p_x^{\alpha}u:0\le|\alpha|\le k\}$ and $\p_t^ku(0,x)=F_k(J_ku_0,J_{k-1}u_1)$ ($0\le k\le 2N$),
where $F_k$ depends on $Q(\p u, \p^2u)$, $J_ku_0$ and $J_{k-1}u_1$.
We say that the compatibility condition for problem \eqref{QWE} up to $(2N+1)-$order holds if
\begin{equation}\label{compat:condition}
\text{$\cB F_k$ vanish on $[0,+\infty)\times\p\cK$ for $0\le k\le 2N$.}
\end{equation}
On the other hand, the following admissible condition of $Q(\p u,\p^2u)$ for Neumann boundary value problem \eqref{QWE}
should be imposed as in (1.8) of \cite{MSS}
\begin{equation}\label{admiss:condition}
\begin{split}
&\text{For any $C^1(\R^2)$ smooth functions $w$ and $\psi$ verifying $\frac{\p}{\p\bn}w=\frac{\p}{\p\bn}\psi=0$ on $[0,+\infty)\times\p\cK$,}\\
&\text{$\sum_{\alpha,\beta,\mu=0}^2Q^{\alpha\beta\mu}n_{\alpha}\p_{\beta}w\p_{\mu}\psi\equiv 0$ holds on $[0,+\infty)\times\p\cK$.}
\end{split}
\end{equation}
Without loss of generality, the initial data $(u_0(\ve,x),u_1(\ve,x))$ are written as
\begin{equation}\label{initial:data}
u_k(\ve,x)=\ve u_k^1(x)+\ve^2u_k^2(\ve,x),\quad k=0,1,
\end{equation}
where $u_k^1(x)\in C^\infty(\cK)$, $\supp_x u_k^1\subset\{x\in\cK:|x|\le M_0\}$ ($k=0,1$),
$(u_0^1(x),  u_1^1(x))\not\equiv 0$, $(u_0^2,u_1^2)(\ve,x)\in C^\infty([0,1)\times\cK)$,
and $\supp_x(u_0^2,u_1^2)\subset\{x\in\cK:|x|\le M_0\}$.

Let $w_{B}=w_{B}(t,x)$ be a solution of the following homogeneous linear IBVP
\begin{equation}\label{LWE}
\left\{
\begin{aligned}
&\Box w_{B}=0,\qquad   \qquad \qquad  \qquad (t,x)\in[0,+\infty)\times\cK,\\
&\cB w_{B}|_{\p\cK}=0, \qquad   \qquad \qquad\quad (t,x)\in[0,+\infty)\times\p\cK,\\
&(w_{B},\p_tw_{B})(0,x)=(u_0^1,u_1^1)(x), \qquad \quad x\in\cK.
\end{aligned}
\right.
\end{equation}
Set $\sigma=r-t$ and define the Friedlander radiation field for problem \eqref{LWE}
\begin{equation}\label{def:F0}
F_0(\sigma,\theta)=\lim_{r\rightarrow+\infty}\sqrt{r}w_{B}(r-\sigma,r\omega).
\end{equation}

\begin{theorem}\label{thm1}
Suppose that  $N\ge19$ and \eqref{compat:condition} hold. In addition, the admissible condition \eqref{admiss:condition}
is imposed  for the Neumann boundary condition. Then
the lifespan $T_\ve$ of the smooth solution $u\in C([0,T_{\ve}), H^{2N+1}(\cK))\cap C^1([0,T_{\ve}), H^{2N}(\cK))$
to problem \eqref{QWE} satisfies
\begin{equation}\label{thm1:lifespan}
\liminf_{\ve\rightarrow 0+}\ve\sqrt{T_\ve}\ge\tau_0=\big(\max\limits_{(\sigma,\theta)}Q_{1st}(\tilde\omega)\p_\sigma^2F_0(\sigma,\theta)\big)^{-1}>0.
\end{equation}

\end{theorem}

\begin{theorem}\label{thm2}
Under the same conditions of Theorem \ref{thm1},
for the problem
\begin{equation}\label{rad:wave}
\left\{
\begin{aligned}
&\p_t^2u-c^2(\p_tu)\Delta u=0,\hspace{2.6cm} (t,x)\in(0,+\infty)\times\{x: |x|\ge1\},\\
&\cB u(t,x)=0,\hspace{4cm} (t,x)\in(0,+\infty)\times\{x: |x|=1\},\\
&(u,\p_tu)(0,x)=(u_0,u_1)(\ve,x)=(u_0,u_1)(\ve,|x|),\qquad\qquad  |x|\ge1
\end{aligned}
\right.
\end{equation}
with $c(0)=1$ and $c'(0)\neq0$, the lifespan $T_\ve$ of the radial smooth solution $u\in C([0,T_{\ve}), H^{2N+1}(\cK))\cap C^1([0,T_{\ve}), H^{2N}(\cK))$ to \eqref{rad:wave} satisfies
\begin{equation}\label{thm2:lifespan}
\limsup_{\ve\rightarrow 0+}\ve\sqrt{T_\ve}\le\tau_0=\big(\max\limits_{\sigma}2c'(0)\p_\sigma^2F_0(\sigma)\big)^{-1},
\end{equation}
where $F_0(\sigma)=F_0(\sigma,\theta)=\ds\lim_{r\rightarrow+\infty}\sqrt{r}w_{B}(r-\sigma,r)$ is independent of
$\theta$ and $w_{B}(t,r)$ corresponds to the radial symmetric solution of \eqref{LWE}
with the initial data $(w_{B},\p_tw_{B})(0,x)=(u_0^1,u_1^1)(r)$.
\end{theorem}

\begin{remark}
For the existence of $F_0(\sigma,\theta)$ in \eqref{def:F0} and the positivity of $\tau_0$ in \eqref{thm1:lifespan},
one can see Proposition \ref{prop31} and Remark \ref{rmk:tau} below, respectively.
\end{remark}

\begin{remark}\label{Rem-1.2}
Note that $Q_{1st}(\tilde\omega)=2c'(0)$ for problem \eqref{rad:wave}.
Therefore, $\tau_0$ in \eqref{thm2:lifespan} just coincides with that in \eqref{thm1:lifespan}
for the radial solution of \eqref{rad:wave}.
This implies $\ds\lim_{\ve\rightarrow 0+}\ve\sqrt{T_\ve}=\tau_0$.

On the other hand, for the general problem \eqref{QWE}, if there exists a unique point $(\sigma_0, \theta_0)$
such that $\tau_0=Q_{1st}(\tilde\omega_0)\p_\sigma^2F_0(\sigma_0,\theta_0
)=\max\limits_{(\sigma,\theta)}Q_{1st}(\tilde\omega)\p_\sigma^2F_0(\sigma,\theta)$
with $\tilde\omega_0=(-1, cos\theta_0, sin\theta_0)$ and the following generic non-degenerate condition holds
\begin{equation}\label{generic}
\nabla_{\sigma,\theta}\big(Q_{1st}(\tilde\omega)\p_\sigma^2F_0(\sigma,\theta)\big)|_{(\sigma,\theta)=(\sigma_0,\theta_0)}=0,\quad
\nabla^2_{\sigma,\theta}\big(Q_{1st}(\tilde\omega)\p_\sigma^2F_0(\sigma,\theta)\big)|_{(\sigma,\theta)=(\sigma_0,\theta_0)}<0,
\end{equation}
then as in the proof of \cite{Alinhac99} together with the basic techniques in the paper,
one can establish that $\ds\lim_{\ve\rightarrow 0+}\ve\sqrt{T_\ve}=\tau_0$
holds for Theorem \ref{thm1} under the condition \eqref{generic}.
\end{remark}

\begin{remark}
When the nonlinearity in \eqref{nonlinear} satisfies the first null condition (i.e., $Q_{1st}(\tilde\omega)\equiv0$),
the global existence of small data
smooth solutions of problem  \eqref{QWE} has been studied in \cite{HouYinYuan24}
and \cite{HouYinYuan26}.
\end{remark}

\begin{remark}
The result in Theorem \ref{thm1} can be directly applied to the initial boundary value problem of 2-D isentropic, irrotational compressible Euler equations and 3-D steady isentropic, irrotational supersonic Euler equations for the polytropic gases in exterior domains with impermeable
boundary conditions, see the detailed introductions for these related models in Section 1 of \cite{HouYinYuan26}.
\end{remark}

\begin{remark}
For problem  \eqref{QWE} with the Dirichlet boundary condition, by establishing a
Morawetz type energy estimate for the 2-D perturbed inhomogeneous wave
equation in the exterior domain, the authors in \cite{LRX25} have
proved that the lifespan $T_{\ve}$ fulfills
\begin{equation}\label{Lai-NA}
T_{\ve}\ge\frac{C}{\ve^2|\ln\ve|^3}.
\end{equation}
It is obvious  that the estimate \eqref{Lai-NA} contains the loss of $|\ln\ve|^3$ than the lifespan ${\bar T}_\ve\ge\frac{C}{\ve^2}$ for the corresponding Cauchy problem (see \cite[Theorem 6.5.8]{Hormander97book}).
Note that the Dirichlet boundary condition in \eqref{nonlinear} has played a key role in the
proof procedure of \cite{LRX25} (see Lemma 3.1 of \cite{LRX25}).
On the other hand, the sharp upper bounds of ${\bar T}_\ve$ for the 2-D Cauchy problem and $T_\ve$
for the 2-D radial exterior domain problem have been obtained in \cite{Alinhac99} and \eqref{thm2:lifespan} in Theorem \ref{thm2}, respectively.
Therefore, as explained in Remark \ref{Rem-1.2}, the estimate of $T_{\ve}$ in Theorem \ref{thm1} is actually optimal.
\end{remark}

\begin{remark}
Note that the cubic and higher order nonlinearities in the more general quadratic form $Q(\p u, \p^2 u)$
do not affect the limit estimate of $T_{\ve}$ in \eqref{thm1:lifespan}.
Then Theorem \ref{thm1} still holds for 
\begin{equation}\label{genra:NL}
Q(\p u,\p^2u)=\sum_{\alpha,\beta,\mu=0}^2Q^{\alpha\beta\mu}\p^2_{\alpha\beta}u\p_{\mu}u
+\sum_{\alpha,\beta,\mu,\nu=0}^2Q^{\alpha\beta\mu\nu}\p^2_{\alpha\beta}u\p_{\mu}u\p_{\nu}u+O(|(\p u,\p^2u)|^4),
\end{equation}
where $Q^{\alpha\beta\mu\nu}$ are constants and $Q^{\alpha\beta\mu\nu}=Q^{\beta\alpha\mu\nu}=Q^{\alpha\beta\nu\mu}$.
Set
\begin{equation}\label{2nd:NC}
Q_{2nd}(\tilde\omega)=\sum_{\alpha,\beta,\mu,\nu=0}^2Q^{\alpha\beta\mu\nu}\omega_\alpha\omega_\beta\omega_\mu\omega_\nu.
\end{equation}
If $Q_{1st}(\tilde\omega)\equiv0$ and $Q_{2nd}(\tilde\omega)\not\equiv0$,
as in the corresponding Cauchy problem (see \cite[Theorem 1]{Alinhac01a}),
one can guess that problem \eqref{QWE} with the nonlinearity $Q(\p u,\p^2u)$ in \eqref{genra:NL} admits such a sharp lifespan
$T_{\ve}$ satisfying
\begin{equation}\label{rmk:tau0}
\liminf_{\ve\rightarrow 0+}\ve^2\ln{T_\ve}\ge\tilde\tau_0=\big(\max\limits_{(\sigma,\theta)}Q_{2nd}(\tilde\omega)\p_\sigma F_0(\sigma,\theta)\p_\sigma^2F_0(\sigma,\theta)\big)^{-1}>0.
\end{equation}
This will be given in our forthcoming paper by rather different methods from those in the present paper
because the estimates here are not available directly.
\end{remark}

\subsection{Previous results}

For the quadratic quasilinear wave equations, consider the following Cauchy problem or the initial boundary value problem in exterior domains
\begin{equation}\label{QWE:Cauchy}
\left\{
\begin{aligned}
&\Box w+\sum_{\alpha,\beta,\mu=0}^nQ^{\alpha\beta\mu}\p^2_{\alpha\beta}w\p_{\mu}w
+\sum_{\al,\beta,\mu,\nu=0}^nQ^{\alpha\beta\mu\nu}\p^2_{\alpha\beta}w\p_{\mu}w\p_{\nu}w=0 ,\quad t>0,\, x\in\R^n,\\
&(w,\p_tw)(0,x)=(\ve w^C_0,\ve w^C_1)(x),\qquad \qquad \qquad \qquad \qquad \qquad \qquad \qquad x\in\R^n
\end{aligned}
\right.
\end{equation}
or
\begin{equation}\label{QWE:exterior}
\left\{
\begin{aligned}
&\Box w+\sum_{\alpha,\beta,\mu=0}^nQ^{\alpha\beta\mu}\p^2_{\al\beta}w\p_{\mu}w
+\sum_{\alpha,\beta,\mu,\nu=0}^nQ^{\alpha\beta\mu\nu}\p^2_{\alpha\beta}w\p_{\mu}w\p_{\nu}w=0 ,\quad t>0,\, x\in\cK,\\
&\cB w=0,\qquad \qquad \qquad \qquad \qquad \qquad \qquad \qquad \qquad \qquad \qquad t>0, x\in\p\cK,\\
&(w,\p_tw)(0,x)=(\ve w^B_0,\ve w^B_1)(x),\qquad \qquad \qquad \qquad \qquad \qquad \qquad x\in\cK,
\end{aligned}
\right.
\end{equation}
where $(w^C_0,w^C_1)\in C_0^{\infty}(\R^n)$, $n\ge 2$, $(w^B_0,w^B_1)\in C_0^{\infty}(\cK)$, $\ve>0$ is small,
$\cK=\R^n\setminus\cO$, the obstacle $\cO\subset\R^n$ is compactly convex and contains the origin, the boundary $\p\cK=\p\cO$ is smooth,
$\cB$ represents ${\rm Id}$ or $\frac{\p}{\p\bn}$ for the Dirichlet or the Neumann boundary conditions, respectively,
and $\bn=(0, n_1,\cdot\cdot\cdot, n_n)$ is the unit outer normal direction of $[0,+\infty)\times\p\cK$.
The lifespans of smooth solutions to problem \eqref{QWE:Cauchy}, the Dirichlet boundary problem or the Neumann boundary problem in \eqref{QWE:exterior} are denoted by $T_{\ve}^{C}$,  $T_{\ve}^{D}$ or $T_{\ve}^{N}$, respectively.
In addition, the related compatible conditions for the initial boundary values in \eqref{QWE:exterior}
hold. On the other hand, for the Neumann boundary condition in \eqref{QWE:exterior}, the corresponding admissible condition \eqref{admiss:condition} is also required. For $n=3$,
set $Q(\tilde\omega)=\ds\sum_{\alpha,\beta,\mu=0}^3Q^{\alpha\beta\mu}\omega_{\alpha}\omega_{\beta}\omega_{\mu}$
with $\tilde\omega=(-1, \omega)$ and $\omega=(\omega_1, \omega_2, \omega_3)\in\Bbb S^2$.
The previous well known conclusions on $T_{\ve}^{C}, T_{\ve}^{D}$ and $T_{\ve}^{N}$
can be roughly listed as follows:
\begin{table}[!h]
\renewcommand\arraystretch{1.2}
\begin{tabular}{|c|c|c|c|c|}
\hline
Dimensions & Quadratic nonlinearity & Cauchy problem & Dirichlet problem & Neumann problem\\
\hline
$n\ge4$ &  & $\ds T_{\ve}^C=\infty$ & $\ds T_{\ve}^D=\infty$ &$\ds T_{\ve}^N=\infty$ \\
\hline
\multirow{2}{*}{$n=3$} & $Q(\tilde\omega)\equiv0$ & $\ds T_{\ve}^C=\infty$ & $\ds T_{\ve}^D=\infty$&$\ds T_{\ve}^N=\infty$ \\
\cline{2-5}
 & $Q(\tilde\omega)\not\equiv0$  & $\ds T_{\ve}^C\ge e^{\frac{C}{\ve}}$ & $\ds T_{\ve}^D\ge e^{\frac{C}{\ve}}$
 & $\ds T_{\ve}^N\ge e^{\frac{C}{\ve}}$
 \\
\hline
\multirow{3}{*}{$n=2$} & $Q_{1st}(\tilde\omega)\equiv0, Q_{2st}(\tilde\omega)\equiv0$ & $\ds T_{\ve}^C=\infty$ & $\ds T_{\ve}^D=\infty$ &{\color{red}$\ds T_{\ve}^N=\infty$~Open} \\
\cline{2-5}
 & $Q_{1st}(\tilde\omega)\not\equiv0$  & $\ds T_{\ve}^C\ge C/\ve^2$ & {\color{red}$\ds T_{\ve}^D\ge C/\ve^2$~Open} &
{\color{red}$\ds T_{\ve}^N\ge C/\ve^2$~Open} \\
\cline{2-5}
 & $Q_{1st}(\tilde\omega)\equiv0, Q_{2st}(\tilde\omega)\not\equiv0$ & $\ds T_{\ve}^C\ge e^{C/\ve^2}$ & {\color{red}$\ds T_{\ve}^D\ge e^{C/\ve^2}$~Open}&
 {\color{red}$\ds T_{\ve}^N\ge e^{C/\ve^2}$~Open} \\
\hline
\end{tabular}
\end{table}

Next we give a detailed illustration on the table above.

$\bullet$  When $n\ge4$, it has been known that both the problems \eqref{QWE:Cauchy} and \eqref{QWE:exterior} admit global smooth
small data  solutions, i.e., $T_{\ve}^{C}, T_{\ve}^{D}, T_{\ve}^{N}=\infty$, one can see \cite{Hormander97book,Klainerman80,KP83,Li-Chen} for the Cauchy problem \eqref{QWE:Cauchy} and \cite{Chen89,Hayashi95,MetcalfeSogge06,Shib-Tsuti85,Shib-Tsuti86} for the Dirichlet/Neumann boundary problem \eqref{QWE:exterior}, respectively.

\vskip 0.2 true cm

$\bullet$ When $n=3$ and the null condition fails (i.e., $Q(\tilde\omega)\not\equiv0$), the sharp lower bounds $T^C_{\ve}, T_{\ve}^D, T_{\ve}^N\ge e^{C/\ve}$ have been shown
for problem \eqref{QWE:Cauchy} (see \cite{JohnKlainerman84}) and problem \eqref{QWE:exterior} (see \cite{Godin89,Godin08,KSS04,MetcalfeSogge06,Yuan22}), respectively.
Moreover, the estimate of the lifespan like \eqref{thm1:lifespan} for the 3-D initial boundary value problem \eqref{QWE:exterior}
has also been established in \cite{Godin08,Godin09,KK13}.
When the null condition holds (that is, $Q(\tilde\omega)\equiv0$), the global existence of small data smooth solutions
has been obtained, see \cite{Christodoulou86,Klainerman} for \eqref{QWE:Cauchy} and \cite{Godin95,KK08,KSS00,Li-Yin18,MetcalfeSogge05} for \eqref{QWE:exterior}, respectively.

\vskip 0.2 true cm

$\bullet$ When $n=2$, $Q^{\alpha\beta\mu}\equiv0$ for all $\alpha,\beta,\mu=0,1,2$ and
the second null condition holds for \eqref{QWE:Cauchy} and \eqref{QWE:exterior} (that is, $Q_{2nd}(\tilde\omega)\equiv0$),
the small data smooth solutions exist globally, see \cite{Katayama95} for \eqref{QWE:Cauchy} and \cite{Kubo13,HouYinYuan26jde} for \eqref{QWE:exterior}, respectively.
When $Q^{\alpha\beta\mu}\equiv0$ for all $\alpha,\beta,\mu=0,1,2$ and
the second null condition fails (i.e., $Q_{2nd}(\tilde\omega)\not\equiv0$) for \eqref{QWE:Cauchy} and \eqref{QWE:exterior},
the estimates on the lower bound of the lifespan have also been studied, see \cite{Godin93,Kovalyov87,Hoshiga95}
for $T_{\ve}^{C}\ge e^{C/\ve^2}$, \cite{Kubo14} for $T_{\ve}^{D}\ge e^{C/\ve^2}$ and \cite{KKL13} for $T_{\ve}^{N}\ge e^{C/\ve}$, respectively.

\vskip 0.2 true cm

$\bullet$ When $n=2$ and $Q^{\alpha\beta\mu}\neq0$ for some $\alpha,\beta,\mu=0,1,2$, the lifespan verifies$\ds\lim_{\ve\rightarrow 0+}\ve\sqrt{T_\ve^C}=\tau_0$ if $Q_{1st}(\tilde\omega)\not\equiv0$ in \eqref{1st:fail}, $\ds\liminf_{\ve\rightarrow 0+}\ve^2\ln{T_\ve^C}=\tilde\tau_0$ if $Q_{1st}(\tilde\omega)\equiv0$ but $Q_{2nd}(\tilde\omega)\not\equiv0$ and $T^C_{\ve}=\infty$ if $Q_{1st}(\tilde\omega)\equiv0$ and $Q_{2nd}(\tilde\omega)\equiv0$, where $\tau_0$ and $\tilde\tau_0$ are given in \eqref{thm1:lifespan} and \eqref{rmk:tau0}, respectively, see \cite{Alinhac99,Alinhac01a,Alinhac01b,Hormander97book}.
For problem \eqref{QWE:exterior}, $T^D_{\ve}=\infty$ has been obtained in \cite{HouYinYuan24} with both null conditions (i.e., $Q_{1st}(\tilde\omega)\equiv0$ and $Q_{2nd}(\tilde\omega)\equiv0$).
If the first null condition fails ($Q_{1st}(\tilde\omega)\not\equiv0$), the authors in \cite{LRX25} show $T^D_\ve\ge\frac{C}{\ve^2|\ln \ve|^3}$
for problem \eqref{QWE:exterior}, which contains $|\ln \ve|^{-3}$ loss than the corresponding Cauchy problem.
Recently, for the 2-D Cauchy-Neumann problem \eqref{QWE:exterior},
the authors in \cite{HouYinYuan26} have proved that $T^N_\ve\ge\frac{C}{\ve^{2-}}$ for $Q_{1st}(\tilde\omega)\not\equiv0$,
$T^N_\ve=\infty$ for the quadratic $Q_0$-type null nonlinearity and $Q_{2nd}(\tilde\omega)\equiv0$.

\vskip 0.2 true cm

The aim of this paper is to improve the estimates of the lifespan in \cite{HouYinYuan26} and \cite{LRX25} for $Q_{1st}(\tilde\omega)\not\equiv0$
and further complete the proof on the open problems of $\ds T_{\ve}^D\ge C/\ve^2$ and $\ds T_{\ve}^N\ge C/\ve^2$.

\vskip 0.2 true cm
\subsection{Sketches of proofs}

Next we give some comments on the proofs of Theorem \ref{thm1} and Theorem \ref{thm2}.
As pointed in \cite{KSS00}-\cite{KSS04} and \cite{HouYinYuan24,HouYinYuan26}, the crucial Lorentz boost fields $t\p_i+x_i\p_t$ $(i=1,2)$
and the scaling vector field $t \p_t+x_1\p_1+x_2\p_2$ can not be used to treat the 2-D initial boundary value problem \eqref{QWE} directly
due to the appearance of the boundary $[0,+\infty)\times\p\cK$. Meanwhile, the lower time decay rate $(1+t)^{-1}\ln(2+t)$
of the local energy for the 2-D wave equation
near the boundary is of another essential difficulty in obtaining the sharp lifespan of problem \eqref{QWE}
(see (1.10a) of Theorem 1.2 in \cite{HouYinYuan26}), which is essentially different from the
3-D problem  with the exponential decay rate of the local energy (see \cite[Lemma 4.3]{KSS04} and \cite[Lemma 2.8]{MetcalfeSogge05}).
On the other hand, the fundamental KSS estimates in 3-D problem \cite{KSS04,MetcalfeSogge05,MetcalfeSogge06} fail in treating
the related 2-D problem, one can refer to \cite[Page 485]{Hepd-Metc23} for detailed discussions.

It is noteworthy that for the proof of  $T^D_{\ve}\ge\frac{C}{\ve^2|\ln\ve|^3}$ in \cite{LRX25} by the energy method,
the key estimate in Lemma 3.1 of \cite{LRX25} depends heavily on the Dirichlet boundary condition
and can not be utilized to study the corresponding Cauchy-Neumann problem \eqref{QWE} directly
(see page 10 of \cite{LRX25} for details), moreover, only $T^D_{\ve}\ge\frac{C}{\ve^2|\ln\ve|^3}$
rather than $T^D_{\ve}\ge\frac{C}{\ve^2}$ is obtained in \cite{LRX25}.
In addition, through establishing a series of $L^\infty-L^\infty$ estimates and energy estimates
for the Neumann-Cauchy problem in \eqref{QWE} (see Sections 3-4 in \cite{HouYinYuan26}),
the authors in \cite{HouYinYuan26} only prove $T^N_{\ve}\ge\frac{C}{\ve^{2-}}$.
In order to prove $T^D_{\ve}\ge\frac{C}{\ve^{2}}$ and $T^N_{\ve}\ge\frac{C}{\ve^{2}}$ in Theorem \ref{thm1},
it is required to introduce new techniques and ideas. Motivated by the proof of $T^C_{\ve}\ge\frac{C}{\ve^{2}}$
for problem \eqref{QWE:Cauchy} in Section 6.5 of \cite[Chapter 6]{Hormander97book},
we will take the following ingredients.

\vskip 0.2 true cm
{\bf $\bullet$ Introduction of the Friedlander radiation field for the IBVP of 2-D linear wave equation}
\vskip 0.1 true cm
Consider the Cauchy problem of 2-D linear wave equation
\begin{equation}\label{YH-2}
\left\{
\begin{aligned}
&\Box w_{C}=0,\qquad   \qquad \qquad  \qquad (t,x)\in[0,+\infty)\times\Bbb R^2,\\
&(w_{C},\p_tw_{C})(0,x)=(u_0^1,u_1^1)(x), \qquad \quad x\in\Bbb R^2.
\end{aligned}
\right.
\end{equation}
Set
\begin{equation*}\label{loc:ibvp:pf3}
\chi_+^{-\f12}(s)=\left\{
\begin{aligned}
&\ds\f{s^{-\f12}}{\Gamma(\f12)}\quad\text{for $s>0$},\\
&0,\quad \quad\text{for $s\le 0$}.
\end{aligned}
\right.
\end{equation*}
The Radon transformation of function $f(x)$ is defined as
$$R(s, \theta; f)=\int_{<\omega,y>=s}f(y)dS(y).$$
Define the Friedlander radiation field for \eqref{YH-2} as
\begin{equation}\label{YH-3}
\tilde F_0(\sigma, \theta)=\f{1}{2\sqrt{2\pi}}\chi_+^{-\f12}*(R(\cdot, \theta; u_1^1)-R(\cdot, \theta; u_0^1)'),
\end{equation}
where the convolution and differentiation are taken in the variable $\sigma$ indicated by a dot.
It follows from Theorem 6.2.1 of \cite{Hormander97book} that for $r\ge ct+1$, the fixed constant $c\in(0,1)$, $t\ge 1$ and $\sigma=r-t\in(-\infty,M_0]$,
\begin{equation}\label{YH-4}
|L^I(w_C(t,x)-r^{-\f12}\tilde F_0(\sigma, \theta))|\le C_{I}(1+|\sigma|)^{\f12}t^{-\f32},
\end{equation}
where $L\in\{\p, t \p_t+x_1\p_1+x_2\p_2, t\p_1+x_1\p_t, t\p_2+x_2\p_t, x_1\p_2-x_2\p_1\}$.
It is pointed out that the property of $\tilde F_0(\theta, \sigma)$ plays a key role in determining
the sharp lifespan $T_{\ve}^C$ of problem \eqref{QWE:Cauchy} when $Q_{1st}(\tilde\omega)\not\equiv0$. That is, in terms of Theorem 6.5.8 of \cite{Hormander97book}
and the result in \cite{Alinhac99}, it holds that
\begin{equation}\label{YH-5}
\lim_{\ve\rightarrow 0+}\ve\sqrt{T_\ve^C}=\tilde\tau_0=\big(\max\limits_{(\sigma,\theta)}Q_{1st}(\tilde\omega)\p_\sigma^2\tilde F_0(\sigma, \theta)\big)^{-1}>0.
\end{equation}
Note that by \eqref{YH-4}, $\tilde F_0(\sigma, \theta)$ can be also determined by
\begin{equation}\label{YH-6}
\tilde F_0(\sigma,\theta)=\lim_{r\rightarrow+\infty}\sqrt{r}w_{C}(r-\sigma,r\omega).
\end{equation}
Motivated by \eqref{YH-6}, for the linear IBVP \eqref{LWE}, we formally define
the Friedlander radiation field for \eqref{LWE} as
\begin{equation}\label{def:F0-1}
F_0(\sigma,\theta)=\lim_{r\rightarrow+\infty}\sqrt{r}w_{B}(r-\sigma,r\omega).
\end{equation}
Indeed, the existence of $F_0(\sigma,\theta)$ in \eqref{def:F0-1} can be shown, meanwhile it holds that
for $r\ge ct+1$, the fixed constant $c\in(0,1)$, $t\ge 1$, $\sigma\in(-\infty,M_0]$,
\begin{equation}\label{LWE-error-0}
\begin{split}
|Z^I(w_{B}(t,x)-r^{-1/2}F_0(\sigma,\theta))|&\le C_I t^{-1}\ln(2+t),\\
|\p Z^I(w_{B}(t,x)-r^{-1/2}F_0(\sigma,\theta))|&\le C_I (1+|\sigma|)^{-\f12}t^{-\f32}\ln(2+t),
\end{split}
\end{equation}
where $Z\in\{\p,x_1\p_2-x_2\p_1\}$. One can see the details in Section \ref{sect3}.

By comparing \eqref{LWE-error-0} with \eqref{YH-4} near the light conic surface $\{(t,x): r-t=M_0\}$,
it is easy to know that the corresponding orders of approximation in \eqref{LWE-error-0} are worse due to the appearance  of the large
factor $\ln(2+t)$, moreover, only the partial vector fields  $\{\p,x_1\p_2-x_2\p_1\}$ in $\{\p, t \p_t+x_1\p_1+x_2\p_2, t\p_1+x_1\p_t, t\p_2+x_2\p_t, x_1\p_2-x_2\p_1\}$ can be used in \eqref{LWE-error-0}.

\vskip 0.2 true cm
{\bf $\bullet$ Construction of the approximate solution for problem \eqref{QWE}}
\vskip 0.1 true cm

Let $U(\tau,\sigma,\theta)$ be the solution of the following problem
\begin{equation}\label{profile-eqn-0}
\left\{
\begin{aligned}
&\p^2_{\tau\sigma}U-Q_{1st}(\tilde\omega)\p_{\sigma}U\p^2_{\sigma}U=0,\\
&U(0,\sigma,\theta)=F_0(\sigma,\theta),\\
&\supp_{\sigma} U\subset \{\sigma\leq M_0\}.
\end{aligned}
\right.
\end{equation}
Set $x=r(\cos\theta,\sin\theta)$, $\sigma=r-t$, $\tau=\ve\sqrt{1+t}$ and
\begin{equation}\label{app-solution}
u_{ap}(t,x)=\ve\chi(\ve t)w_{B}(t,x)+\ve(1-\chi(\ve t))\chi(-3\ve\sigma)r^{-1/2}U(\tau,\sigma,\theta),
\end{equation}
where $\chi(s)=1-\chi_{[1,2]}(s)$ with $0\le\chi_{[1,2]}(s)\le1$ and
  $\chi_{[1,2]}(s)=\left\{
    \begin{aligned}
    0,\qquad s\le 1,\\
    1,\qquad s\ge 1.
    \end{aligned}
    \right.$ Meanwhile, $\cB u_{ap}|_{[0,\tau_0)\times \p\cK}=0$ holds.
$u_{ap}(t,x)$ will be taken as the approximate solution of problem \eqref{QWE}.
This is rather analogous to the construction of the approximate solution for the 2-D Cauchy problem in \eqref{QWE:Cauchy}
(see Section 6.5 of \cite{Hormander97book} Chapter 6).

\vskip 0.2 true cm
{\bf $\bullet$ Precise estimates of the error $u-u_{ap}$ for problem \eqref{QWE}}
\vskip 0.1 true cm

Under the suitable bootstrap assumption for $x\in\cK$ and $t\in[0,B^2/\ve^2-1]$ with any fixed positive constant $B<\tau_0$,
\begin{equation}\label{BA-0}
\begin{split}
\sum_{|I|\le N}|\p Z^I(u-u_{ap})|&\le C_I\ve^{\f32-}\w{x}^{-1/2},\\
\sum_{|I|\le N}|\p Z^I(u-u_{ap})|&\le C_I\ve\w{x}^{-1/2}\w{t-|x|}^{-1}\w{t}^{0+},
\end{split}
\end{equation}
where $\w{x}=\sqrt{1+|x|^2}$ and the generic constant $C_I$ depends on $B$, we can arrive at (see Lemma \ref{YH-10} below)
\begin{equation}\label{intro-energy2-0}
\sum_{|I|\le2N-6}\|\p Z^I(u-u_{ap})\|_{L^2(\cK)}\le C_I\ve^{3/2-}.
\end{equation}
Then it follows from \eqref{intro-energy2-0} and the Sobolev embedding for $Z\in\{\p,x_1\p_2-x_2\p_1\}$ that (see Lemma \ref{YH-1})
\begin{equation}\label{intro-decay1-0}
\sum_{|I|\le2N-8}|\p Z^Iu(t,x)|\le C\ve^{3/2-}\w{x}^{-1/2}.
\end{equation}
Note that the field $Z$ in \eqref{intro-energy2-0} does not include $t\p_t+r\p_r$
and $t\p_i+x_i\p_t$ $(i=1,2)$, we can not obtain the suitable space-time decay estimate
from \eqref{intro-energy2-0} for $\p Z^I(u-u_{ap})$ by the following Klainerman-Sobolev inequality
\begin{equation*}
|\phi(t,x)|\le \f{C}{(1+|t-r|)^{\f12}(1+t)^{\f{1}{2}}}\ds\sum_{|I|\le 2}\|L^I\phi(t,x)\|_{L_x^2(\mathbb R^2)},
\end{equation*}
where $L\in\{\p, t \p_t+x_1\p_1+x_2\p_2, t\p_1+x_1\p_t, t\p_2+x_2\p_t, x_1\p_2-x_2\p_1\}$.
To overcome this difficulty, based on \eqref{intro-energy2-0}-\eqref{intro-decay1-0} and the divergence form of the nonlinearity $Q(\p u,\p^2u)$ in \eqref{nonlinear}
\begin{equation}\label{wave:div}
\begin{split}
\sum_{\alpha,\beta,\mu=0}^2Q^{\alpha\beta\mu}\p^2_{\alpha\beta}u\p_{\mu}u
=\frac12\sum_{\alpha,\beta,\mu=0}^2[\p_{\alpha}(Q^{\alpha\beta\mu}\p_{\beta}u\p_{\mu}u)+\p_{\beta}(Q^{\alpha\beta\mu}\p_{\alpha}u\p_{\mu}u)
-\p_{\mu}(Q^{\alpha\beta\mu}\p_{\alpha}u\p_{\beta}u)],
\end{split}
\end{equation}
together with a series of $L^\infty-L^\infty$ estimates and energy estimates,
we eventually obtain the following pointwise spacetime estimate for any positive constant $B<\tau_0$ and $t\le\frac{B^2}{\ve^2}-1$,
\begin{equation}\label{intro-decay4-0}
\sum_{|I|\le N}|\p Z^I(u-u_{ap})(t,x)|\le C\ve\w{x}^{-1/2}\w{t-|x|}^{-1}\w{t}^{0+}.
\end{equation}
This, together with \eqref{intro-decay1-0}, yields the proof for the bootstrap assumption \eqref{BA-0}.

\vskip 0.2 true cm

Through modifying the characteristics method in \cite{John} for treating the blowup of  the radial symmetric solution
to the 3-D quasilinear wave equation $\p_t^2u-c^2(\p_tu)\Delta u=0$ so that the corresponding 2-D radial exterior domain problem \eqref{rad:wave}
can be adapted, we then complete the proof of Theorem \ref{thm2} together with some estimates derived in Theorem \ref{thm1}.

\vskip 0.2 true cm

The paper is organized as follows.
In Section \ref{sect2}, some basic lemmas and useful conclusions on the pointwise spacetime estimates of solutions to 2-D
linear wave equations are listed or shown.
In Section \ref{sect3}, the existence and the estimate of the Friedlander radiation field are established.
In Section \ref{sect4}, the approximate solution is constructed and its related estimates are derived.
In Section \ref{sect5}, by establishing the delicate energy estimates and the local energy decay estimates for the error solution, the key pointwise spacetime estimates are improved and then the proof of Theorem \ref{thm1} is completed by the continuity argument.
In Section \ref{sect6}, Theorem \ref{thm2} is proved.

\vskip 0.2 true cm

\noindent \textbf{Notations:}
\begin{itemize}
  \item $\cK=\R^2\setminus\cO$, $\cK_R=\cK\cap\{x: |x|\le R\}$, $R>1$ is a fixed constant which may be changed in different places.
  \item For technical reasons and without loss of generality, one can assume $\p\cK\subset\{x:c_0<|x|<1/2\}$, where $c_0$ is a positive constant.
  \item The cutoff function $\chi_{[a,b]}(s)\in C^\infty(\R)$ with $a,b\in\R$ and $0<a<b$, $0\le\chi_{[a,b]}(s)\le1$ and
  \begin{equation*}
    \chi_{[a,b]}(s)=\left\{
    \begin{aligned}
    0,\qquad s\le a,\\
    1,\qquad s\ge b.
    \end{aligned}
    \right.
  \end{equation*}
  \item $\w{x}=\sqrt{1+|x|^2}$.
  \item $\tilde\ve=10^{-3}$.
  \item $\N_0=\{0,1,2,\cdots\}$ and $\N=\{1,2,\cdots\}$.
  \item $\p_0=\p_t$, $\p_1=\p_{x_1}$, $\p_2=\p_{x_2}$, $\p_x=\nabla_x=(\p_1,\p_2)$, $\p=(\p_t,\p_1,\p_2)$.
  \item $\Omega=x_1\p_2-x_2\p_1$,  $Z=\{Z_1,Z_2,Z_3,Z_4\}=\{\p_t,\p_1,\p_2,\Omega\}$, $\tilde Z=\{\tilde Z_1,\tilde Z_2,\tilde Z_3,\tilde Z_4\}$, $\tilde Z_1=\p_t$, $\tilde Z_i=\chi_{[1/2,1]}(x)Z_i,i=2,3,4$.
  \item $\p_x^a=\p_1^{a_1}\p_2^{a_2}$ for $a\in\N_0^2$, $\p^a=\p_t^{a_1}\p_1^{a_2}\p_2^{a_3}$ for $a\in\N_0^3$
      and $Z^a=Z_1^{a_1}Z_2^{a_2}Z_2^{a_3}Z_4^{a_4}$ for $a\in\N_0^4$.
  \item The commutator $[X,Y]=XY-YX$.
  \item For $f,g\ge0$, $f\ls g$ represents $f\le C_Bg$ for a generic positive constant $C_B$ depending on $B<\tau_0$ in \eqref{BA-0}.
  \item Denote $\|u(t)\|=\|u(t,\cdot)\|$ with norm $\|\cdot\|=\|\cdot\|_{L^2(\cK)},\|\cdot\|_{L^2(\cK_R)},\|\cdot\|_{H^k(\cK)}$.
  \item $\cW_{\mu,\nu}(t,x)=\w{t+|x|}^\mu(\min\{\w{x},\w{t-|x|}\})^\nu$ for $\mu, \nu\in\Bbb R$.
  \item $\ds|Z^{\le j}f|=\Big(\sum_{0\le|a|\le j}|Z^af|^2\Big)^\frac12$ and $Z^{\le1}f=(f,Zf)$.
\end{itemize}

\section{Preliminaries}\label{sect2}

\subsection{Several lemmas}\label{YH-7}

In this subsection, some useful lemmas on the elliptic estimate, the local energy decay estimate,
the Gronwall's Lemma, the Sobolev embedding and the blowup lemma will be listed.

\begin{lemma}[Lemma 3.2 of \cite{Kubo13}]
Let $w\in H^j(\cK)$ and $\ds \cB w|_{\p\cK}=0$ with integer $j\ge2$.
Then for any fixed constant $R>1$ and multi-index $a\in\N_0^2$ with $2\le|a|\le j$, one has
\begin{equation}\label{ellip}
\|\p_x^aw\|_{L^2(\cK)}\ls\|\Delta w\|_{H^{|a|-2}(\cK)}+\|w\|_{H^{|a|-1}(\cK_{R+1})}.
\end{equation}
\end{lemma}

\begin{lemma}[Lemma 3.5 of \cite{Kubo13}]\label{lem:loc:decay}
Suppose that the obstacle $\cO$ is star-shaped/convex for the Cauchy-Dirichlet/Neumann problem, respectively.
Denote $\cK=\R^2\setminus\cO$ and let $w$ be the solution of the IBVP
\begin{equation}\label{loc:decay:eqn}
\left\{
\begin{aligned}
&\Box w=H(t,x),\qquad (t,x)\in(0,\infty)\times\cK,\\
&\cB w|_{\p\cK}=0,\\
&(w,\p_tw)(0,x)=(w_0,w_1)(x),\quad x\in\cK,
\end{aligned}
\right.
\end{equation}
where $\supp_x(w_0(x),w_1(x),H(t,x))\subset\{x: |x|\le R_1\}$ for $R_1>1$.
Then for fixed constant $R>1$ and $m\in\N$, there is a positive constant $C=C(m,R,R_1)$ such that
\begin{equation}\label{loc:decay}
\sum_{|a|\le m}\w{t}\|\p^aw\|_{L^2(\cK_R)}
\le C(\|w_0\|_{H^{m}(\cK)}+\|w_1\|_{H^{m-1}(\cK)}+\ln(2+t)\sum_{|a|\le m-1}\sup_{0\le s\le t}\w{s}\|\p^aH(s,\cdot)\|_{L^2(\cK)}).
\end{equation}
\end{lemma}

\begin{lemma}[Lemma 2.4 of \cite{HouYinYuan26}]\label{lem:Gronwall}
For any non-negative constants $A,B_1,C,D$ with $B_1>0$ and function $f(t)\ge0$ satisfying
\begin{equation*}
f(t)\le A+B_1\int_0^t\frac{f(s)ds}{(1+s)^{1/2}}+C(1+t)^D,
\end{equation*}
then
\begin{equation*}
f(t)\le e^{2B_1(1+t)^{1/2}}[A+C(1+t)^D].
\end{equation*}
\end{lemma}

\begin{lemma}[Lemma 3.1 of \cite{Kubo13}]\label{YH-1}
Given any function $f(x)\in C_0^2(\overline{\cK})$, one has that for all $x\in\cK$,
\begin{equation}\label{Sobo:ineq}
\w{x}^{1/2}|f(x)|\ls\sum_{|a|\le2}\|Z^af\|_{L^2(\cK)}.
\end{equation}
\end{lemma}

\begin{lemma}[Lemma 1.3.2 of \cite{Hormander97book}]\label{blowup-lem}
Let $y(t)\in C^1[0,T]$ be a solution of 
$\frac{dy}{dt}=a_0(t)y^2+a_1(t)y+a_2(t)$,
where $a_j(t)\in C[0,T]$ $(j=0,1,2)$ and $a_0(t)\ge0$ for all $t\in[0,T]$.
Define 
\begin{equation*}
K=\Big(\int_0^T|a_2(t)|dt\Big)\exp\Big(\int_0^T|a_1(t)|dt\Big).
\end{equation*}
If $y(0)>K$, then it holds that
\begin{equation*}
\Big(\int_0^Ta_0(t)dt\Big)\exp\Big(-\int_0^T|a_1(t)|dt\Big)<(y(0)-K)^{-1}.
\end{equation*}
\end{lemma}

\subsection{Some crucial pointwise estimates of Cauchy problem}\label{YH-8}

\begin{lemma}
Let $w$ be the solution of the Cauchy problem
\begin{equation*}
\left\{
\begin{aligned}
&\Box w=F,\qquad\qquad\qquad\qquad\quad (t,x)\in[0,\infty)\times\R^2,\\
&(w,\p_tw)(0,x)=(w_0,w_1),\qquad x\in\R^2.
\end{aligned}
\right.
\end{equation*}
Then for $\mu,\nu\in(0,1/2)$, one has
\begin{equation}\label{dpw:ivp}
\w{x}^{1/2}\w{t-|x|}^{1+\mu}|\p w|\ls\cA_{4,1}[w_0,w_1]+\sum_{|a|\le1}\sup_{(s,y)\in[0,t]\times\R^2}\w{y}^{1/2}\cW_{1+\mu+\nu,1}(s,y)|Z^aF(s,y)|,
\end{equation}
where $\ds\cA_{\kappa,s}[f,g]:=\sum_{\tilde\Gamma\in\{\p_1,\p_2,\Omega\}}
(\sum_{|a|\le s+1}\|\w{z}^\kappa\tilde\Gamma^af(z)\|_{L^\infty}+\sum_{|a|\le s}\|\w{z}^\kappa\tilde\Gamma^ag(z)\|_{L^\infty})$
and $\ds\cW_{\mu,\nu}(t,x)=\w{t+|x|}^\mu(\min\{\w{x},\w{t-|x|}\})^\nu$.
In addition, for $F\equiv0$ one has
\begin{equation}\label{pw:ivp}
\w{t+|x|}^{1/2}\w{t-|x|}^{1/2}|w|\ls\cA_{3,0}[w_0,w_1].
\end{equation}
\end{lemma}
\begin{proof}
See (3.7), (3.5) of \cite{HouYinYuan24} for \eqref{dpw:ivp} and \eqref{pw:ivp}, respectively.
\end{proof}

\begin{lemma}[Lemmas 3.8 and 3.9 of \cite{HouYinYuan24}]\label{lem:pw:compact}
Let $v$ be the solution of the Cauchy problem
\begin{equation*}
\left\{
\begin{aligned}
&\Box w=H(t,x),\qquad\qquad(t,x)\in[0,\infty)\times\R^2,\\
&(w,\p_tw)(0,x)=(0,0),\quad x\in\R^2,
\end{aligned}
\right.
\end{equation*}
provided that $\supp_x H(t,x)\subset\{x: |x|\le R\}$.
Then it holds that for $\mu\in(0,1/2]$,
\begin{equation}\label{pw:compact}
\w{t+|x|}^{1/2}\w{t-|x|}^{\mu}|w|
\ls\ln^2(2+t+|x|)\sup_{(s,y)\in[0,t]\times\R^2}\w{s}^{1/2+\mu}|(1-\Delta)H(s,y)|
\end{equation}
and
\begin{equation}\label{dpw:compact}
\w{x}^{1/2}\w{t-|x|}|\p w|
\ls\sup_{(s,y)\in[0,t]\times\R^2}\w{s}|(1-\Delta)^3H(s,y)|.
\end{equation}
\end{lemma}

\subsection{Some estimates of the initial boundary value problem}\label{YH-9}

\begin{lemma}
Suppose that the obstacle $\cO$ is star-shaped/convex for the Cauchy-Dirichlet/Neumann problem, respectively.
Set $\cK=\R^2\setminus\cO$ and let $w$ be the solution of the IBVP
\begin{equation*}
\left\{
\begin{aligned}
&\Box w=\sum_{\alpha=0}^2\p_{\alpha}G^\alpha+G^r,\qquad(t,x)\in[0,\infty)\times\cK,\\
&\cB w|_{[0,\infty)\times\p\cK}=0,\\
&(w,\p_tw)(0,x)=(w_0,w_1)(x),\quad x\in\cK,
\end{aligned}
\right.
\end{equation*}
where $(w_0,w_1)$ has compact support and $\supp_x(G^r,G^\alpha)(t,x)\subset\{x:|x|\le t+M_0\}$.
Then it holds that for any $\eta\in(0,1/2)$,
\begin{equation}\label{dpw:ibvp}
\begin{split}
&\quad\w{x}^{1/2}\w{t-|x|}|\p w|\ls\|(w_0,w_1)\|_{H^9(\cK)}+\ln(2+t)\sum_{|a|\le8}\sup_{s\in[0,t]}\w{s}\|\p^a(G^r,\p^{\le1}G^\alpha)(s)\|_{L^2(\cK_3)}\\
&\qquad\qquad+\ln(2+t)[\sum_{|a|\le8}\sup_{y\in\cK}|Z^a(G^r,\p^{\le1}G^\alpha)(0,y)|+\w{t}^\eta\sup_{y\in\cK}|\p^{\le1}G^\alpha(0,y)|]\\
&\qquad\qquad+\w{t}^\eta\ln^2(2+t)\sum_{|a|\le10}\sup_{(s,y)\in[0,t]\times\overline{\R^2\setminus\cK_2}}\w{y}^{1/2}\cW_{1,1}(s,y)|G^\alpha(s,y)|\\
&\qquad\qquad+\ln(2+t)\sum_{|a|\le9}\sup_{(s,y)\in[0,t]\times\overline{\R^2\setminus\cK_2}}\w{y}^{1/2}\cW_{3/2+\eta,1}(s,y)|\p^aG^r(s,y)|.
\end{split}
\end{equation}
In addition, for $G^r\equiv0$ and $R>1$, one has
\begin{equation}\label{loc:ibvp}
\begin{split}
&\w{t}\|\p^{\le1}w\|_{L^2(\cK_R)}
\ls\|(w_0,w_1)\|_{H^1(\cK)}+\w{t}^\eta\ln(2+t)\sup_{y\in\cK}|\p^{\le1}G^\alpha(0,y)|\\
&\quad+\ln(2+t)\sum_{|a|\le1}\sup_{s\in[0,t]}\w{s}\|\p^aG^\alpha(s)\|_{L^2(\cK_3)}\\
&\quad+\w{t}^\eta\ln^2(2+t)\sum_{|a|\le2}\sup_{(s,y)\in[0,t]\times\overline{\R^2\setminus\cK_2}}\w{y}^{1/2}\cW_{1,1}(s,y)|Z^aG^\alpha(s,y)|.
\end{split}
\end{equation}
\end{lemma}
\begin{proof}
For the Cauchy-Neumann problem, \eqref{dpw:ibvp} and \eqref{loc:ibvp} come from (3.1) and (3.11) of \cite{HouYinYuan26}, respectively.
Next we show  \eqref{loc:ibvp} for the Cauchy-Dirichlet problem.
Note that although the related proof is analogous to that in \cite{HouYinYuan26},
we still give the details for the completeness.
As in \cite{HouYinYuan24,HouYinYuan26}, let $w_1^b$ and $w_2^b$ be the solutions of
\begin{equation}\label{loc:ibvp:pf1}
\begin{split}
&\Box w_1^b:=(1-\chi_{[2,3]}(x))(\sum_{\alpha=0}^2\p_{\alpha}G^\alpha)-\sum_{i=1}^2\p_i(\chi_{[2,3]}(x))G^i,\\
&\Box w_2^b:=\sum_{\alpha=0}^2\p_{\alpha}(\chi_{[2,3]}(x)G^\alpha),\\
&w_1^b|_{[0,\infty)\times\p\cK}=w_2^b|_{[0,\infty)\times\p\cK}=0,\\
&(w_1^b,\p_tw_1^b)(0,x)=(w_0,w_1)(x),\quad (w_2^b,\p_tw_2^b)(0,x)=(0,0),
\end{split}
\end{equation}
respectively.
Then the uniqueness of smooth solution to the IBVP ensures that $w=w_1^b+w_2^b$.
For the estimate of $w_1^b$, it follows from $\supp_x\Box w_1^b\subset\{x:|x|\le3\}$ and \eqref{loc:decay} with $m=1$ that
\begin{equation}\label{loc:ibvp:pf2}
\begin{split}
\w{t}\|\p^{\le1}w_1^b\|_{L^2(\cK_R)}
&\ls\|(w_0,w_1)\|_{H^1(\cK)}+\ln(2+t)\sup_{s\in[0,t]}\w{s}\|\Box w_1^b(s)\|_{L^2(\cK)}\\
&\ls\|(w_0,w_1)\|_{H^1(\cK)}+\ln(2+t)\sum_{|a|\le1}\sup_{s\in[0,t]}\w{s}\|\p^aG^\alpha(s)\|_{L^2(\cK_3)}.
\end{split}
\end{equation}
Next we estimate $w_2^b$.
Let $w_2^c$ be the solution of the following Cauchy problem
\begin{equation}\label{loc:ibvp:pf3}
\Box w_2^c=\left\{
\begin{aligned}
&\sum_{\alpha=0}^2\p_{\alpha}(\chi_{[2,3]}(x)G^\alpha),\qquad&& x\in\cK=\R^2\setminus\cO,\\
&0,&& x\in\overline{\cO}
\end{aligned}
\right.
\end{equation}
with $(w_2^c, \p_tw_2^c)(0,x)=(0,0)$.
Define $v_2^b=\chi_{[1,2]}(x)w_2^c$ on $\cK$ and thus $v_2^b$ is the solution of the IBVP
\begin{equation}\label{loc:ibvp:pf4}
\begin{split}
&\Box v_2^b=\chi_{[1,2]}\Box w_2^c-[\Delta,\chi_{[1,2]}]w_2^c=\sum_{\alpha=0}^2\p_{\alpha}(\chi_{[2,3]}(x)G^\alpha)-[\Delta,\chi_{[1,2]}]w_2^c,\\
&v_2^b|_{[0,\infty)\times\p\cK}=0,\quad (v_2^b,\p_tv_2^b)(0,x)=(0,0).
\end{split}
\end{equation}
On the other hand, let $\tilde v_2^b$ be the solution of the IBVP
\begin{equation}\label{loc:ibvp:pf5}
\Box\tilde v_2^b=[\Delta,\chi_{[1,2]}]w_2^c,
\quad\tilde v_2^b|_{[0,\infty)\times\p\cK}=0,\quad (\tilde v_2^b,\p_t\tilde v_2^b)(0,x)=(0,0).
\end{equation}
Consequently, it is deduced from \eqref{loc:ibvp:pf1}, \eqref{loc:ibvp:pf4} and \eqref{loc:ibvp:pf5} that
\begin{equation}\label{loc:ibvp:pf6}
w_2^b=\chi_{[1,2]}(x)w_2^c+\tilde v_2^b.
\end{equation}
Since $\Box\tilde v_2^b$ is supported in $\{x:|x|\le2\}$, similarly to \eqref{loc:ibvp:pf2} for $w_1^b$, one obtains
\begin{equation}\label{loc:ibvp:pf7}
\begin{split}
\w{t}\|\p^{\le1}\tilde v_2^b\|_{L^2(\cK_R)}
&\ls\ln(2+t)\sup_{s\in[0,t]}\w{s}\|[\Delta,\chi_{[1,2]}]w_2^c(s)\|_{L^2(\cK_2)}\\
&\ls\ln(2+t)\sum_{|a|\le1}\sup_{s\in[0,t]}\w{s}\|\p_x^aw_2^c(s)\|_{L^\infty(|y|\le2)}.
\end{split}
\end{equation}
We now deal with $w_2^c$ in the right hand side of \eqref{loc:ibvp:pf7}.
According to the definition \eqref{loc:ibvp:pf3}, one sees that $\ds w_2^c=\sum_{\alpha=0}^2\p_{\alpha}w^{\alpha}+w^r$ with $w^{\alpha}$ and $w^r$ being the solutions to
\begin{equation}\label{loc:ibvp:pf8}
\begin{split}
&\Box w^{\alpha}=\chi_{[2,3]}(x)G^\alpha,\qquad \Box w^r=0,\\
&(w^{\alpha},\p_tw^{\alpha})(0,x)=(0,0),\quad (w^r,\p_tw^r)(0,x)=(0,-\chi_{[2,3]}(x)G^\alpha(0,x)),
\end{split}
\end{equation}
respectively.
Applying \eqref{dpw:ivp} to $w^{\alpha}$ and \eqref{pw:ivp} to $w^r$ with $\mu=\nu=\eta/2$, respectively, we have
\begin{equation}\label{loc:ibvp:pf9}
\begin{split}
&\w{x}^{1/2}\w{t-|x|}^{1+\eta/2}|\p^{\le1}\p w^{\alpha}|\ls\sum_{\alpha=0}^2\sup_{y\in\cK}|G^\alpha(0,y)|\\
&\qquad+\sum_{\alpha=0}^2\sum_{|a|\le2}\sup_{(s,y)\in[0,t]\times\overline{\R^2\setminus\cK_2}}\w{y}^{1/2}\cW_{1+\eta,1}(s,y)|Z^aG^\alpha(s,y)|,\\
&\w{t+|x|}^{1/2}\w{t-|x|}^{1/2}|\p^{\le1}w^r|\ls\sum_{\alpha=0}^2\sup_{y\in\cK}|G^{\alpha}(0,y)|,
\end{split}
\end{equation}
where the term $\sup_{y\in\cK}|G^\alpha(0,y)|$ in the first line of \eqref{loc:ibvp:pf9} comes from the initial data of $\p_tw^{\alpha}$.
Therefore, collecting \eqref{loc:ibvp:pf2}, \eqref{loc:ibvp:pf6} and \eqref{loc:ibvp:pf7}-\eqref{loc:ibvp:pf9} yields \eqref{loc:ibvp}.
\end{proof}

\section{Asymptotic behavior of the solution to linear IBVP}\label{sect3}

In this section, we will show the existence of the Friedlander radiation filed $F_0(\sigma,\theta)$ in \eqref{def:F0} and derive
some related estimates.
The main estimates in this section are stated as follows.
\begin{proposition}\label{prop31}
Let $w_{B}$ be the solution to the homogeneous linear IBVP \eqref{LWE}.
Then there exits a smooth function $F_0(\sigma,\theta)$ such that for $r\ge ct+1$, $c\in(0,1)$, $\sigma\in(-\infty,M_0]$,
\begin{equation}\label{LWE-error}
\begin{split}
|Z^a(w_{B}(t,x)-r^{-1/2}F_0(\sigma,\theta))|&\ls\w{t}^{-1}\ln(2+t),\\
|\p Z^a(w_{B}(t,x)-r^{-1/2}F_0(\sigma,\theta))|&\ls\w{t}^{-3/2}\w{\sigma}^{-1/2}\ln(2+|t)
\end{split}
\end{equation}
and
\begin{equation}\label{FRF-decay}
|Z^aF_0(\sigma,\theta)|\ls\w{\sigma}^{-1/2}\ln(2+|\sigma|),\quad |\p Z^aF_0(\sigma,\theta)|\ls\w{\sigma}^{-1}.
\end{equation}
\end{proposition}

To prove Proposition \ref{prop31}, it suffices to estimate $w_{B}$ in the region $|x|\ge1$.
To this end, let $w^c(t,x)=\chi_{[1/2,1]}(x)w_{B}(t,x)$ be a function on $\R^2$, which satisfies
\begin{equation}\label{prop31:pf1}
\Box w^c=-[\Delta,\chi_{[1/2,1]}]w_{B}
\end{equation}
with the initial data $(w^c,\p_tw^c)(0,x)=(\chi_{[1/2,1]}u_0^1,\chi_{[1/2,1]}u_1^1)(x)$.
Note that $w^c$ can be decomposed into two parts: $w^c=w^c_1+w^c_2$, such that
\begin{equation}\label{prop31:pf2}
\begin{split}
&\Box w^c_1=-[\Delta,\chi_{[1/2,1]}]w_{B}, \quad (w^c_1,\p_tw^c_1)(0,x)=(0,0),\\
&\Box w^c_2=0, \quad (w^c_2,\p_tw^c_2)(0,x)=(\chi_{[1/2,1]}u_0^1,\chi_{[1/2,1]}u_1^1)(x).
\end{split}
\end{equation}
At first, we focus on the estimate of $w^c_1$.
\begin{lemma}
Let $w_{B}$ and $w_1^c$ be the solutions to \eqref{LWE} and \eqref{prop31:pf2}, respectively.
Then there exits a smooth function $F_0^1(\sigma,\theta)$ such that for $r\ge ct+1$, $c\in(0,1)$,
\begin{equation}\label{LWE-err1}
\begin{split}
|Z^a(w^c_1(t,x)-r^{-1/2}F_0^1(\sigma,\theta))|&\ls\w{t}^{-1}\ln(2+t),\\
|\p Z^a(w_{1}^c(t,x)-r^{-1/2}F_0^1(\sigma,\theta))|&\ls\w{t}^{-3/2}\w{\sigma}^{-1/2}\ln(2+t)
\end{split}
\end{equation}
and
\begin{equation}\label{FRF1-decay}
|\p_{\theta}^jF_0^1(\sigma,\theta)|\ls\w{\sigma}^{-1/2}\ln(2+|\sigma|),\quad |\p_{\sigma}^i\p_{\theta}^jF_0^1(\sigma,\theta)|\ls\w{\sigma}^{-1},\quad i\ge1, j\ge 0.
\end{equation}
\end{lemma}
\begin{proof}
Note that
\begin{equation}\label{LWE-err-pf1}
w^c_1(t,x)=\int_0^t\int_{\R^2}E(t-\tau,x-y)G(\tau,y)dyd\tau,
\end{equation}
where
\begin{equation}\label{LWE-err-pf2}
\begin{split}
& G=-[\Delta,\chi_{[1/2,1]}]w_{B},\qquad E(t,x)=\frac{1}{2\sqrt{\pi}}\chi_+^{-1/2}(t^2-|x|^2),\\
& \chi_+^{-1/2}(s)=s^{-1/2}/\Gamma(\frac{1}{2}),\quad s>0;\quad \chi_+^{-1/2}(s)=0, \quad s\le0.
\end{split}
\end{equation}
Due to $\supp_yG(\tau,y)\subset\{y: |y|\le1\}$, $\supp_x w_1^c(t,x)$ is included in $\{x: |x|\le1+t\}$.
For $x=r\omega$ and $\sigma=r-t$, one has
\begin{equation*}
(t-\tau)^2-|x-y|^2=2r(y\cdot\omega-\sigma-\tau)+(\sigma+\tau)^2-|y|^2.
\end{equation*}
Then it follows from \eqref{LWE-err-pf1}, \eqref{LWE-err-pf2} and the homogeneity of $\chi_+^{-1/2}$ that
\begin{equation}\label{LWE-err-pf3}
\begin{split}
&2\sqrt{2\pi}\sqrt{r}w^c_1(t,x)=\int_0^t\int_{\R^2}\chi_+^{-1/2}(y\cdot\omega-\sigma-\tau+\frac{(\sigma+\tau)^2-|y|^2}{2r})G(\tau,y)dyd\tau\\
&=\int_0^t\int_{\R^2}\int\chi_+^{-1/2}(s-\sigma-\tau+\frac{(\sigma+\tau)^2}{2r})\delta(s-y\cdot\omega+\frac{|y|^2}{2r})G(\tau,y)dsdyd\tau\\
&=2\sqrt{2\pi}F(\sigma,\theta,r^{-1}),
\end{split}
\end{equation}
where $\delta(\cdot)$ is Dirac function, $F(\sigma,\theta,z)$ is supported in $\{\sigma\le M_0\}$, and
\begin{equation}\label{LWE-err-pf4}
\begin{split}
F(\sigma,\theta,z)&=\frac{1}{2\sqrt{2\pi}}\int_0^{1/z-\sigma}\int\chi_+^{-1/2}(s-\sigma-\tau+\frac{z(\sigma+\tau)^2}{2})G(\tau,\theta,s,z)dsd\tau,\\
G(\tau,\theta,s,z)&=-\int_{\R^2}\delta(s-y\cdot\omega+\frac{|y|^2z}{2})[\Delta,\chi_{[1/2,1]}]w_{B}(\tau,y)dy.\\
\end{split}
\end{equation}
Define
\begin{equation}\label{LWE-err-pf5}
\begin{split}
F_0^1(\sigma,\theta)&=F(\sigma,\theta,0+)=\frac{1}{2\sqrt{2\pi}}\int_0^\infty\int\chi_+^{-1/2}(s-\sigma-\tau)G(\tau,\theta,s,0)dsd\tau,
\end{split}
\end{equation}
where $G(\tau,\theta,s,0)=-\int_{\R^2}\delta(s-y\cdot\omega)[\Delta,\chi_{[1/2,1]}]w_{B}(\tau,y)dy.$

At first, we show that $F_0^1(\sigma,\theta)$ is well-defined.
By the homogeneity of $\chi_+^{-1/2}$ and \eqref{loc:decay} with $H=0$, one can achieve
\begin{equation}\label{LWE-err-pf6}
\begin{split}
&|F_0^1(\sigma,\theta)|\ls\int_0^\infty\w{\sigma+\tau}^{-1/2}\w{\tau}^{-1}d\tau\\
&\ls\int_{2|\sigma|}^\infty\w{\tau}^{-3/2}d\tau+\w{\sigma}^{-1}\int_{|\sigma|/2}^{2|\sigma|}\w{\sigma+\tau}^{-1/2}d\tau
+\w{\sigma}^{-1/2}\int_0^{|\sigma|/2}\w{\tau}^{-1}d\tau\\
&\ls\w{\sigma}^{-1/2}\ln(2+|\sigma|),
\end{split}
\end{equation}
which yields the first inequality of \eqref{FRF1-decay} with $j=0$.

Next, we prove the second inequality of \eqref{FRF1-decay} with $i=1,j=0$. Note that
\begin{equation*}
\begin{split}
&|\p_{\sigma}F_0^1(\sigma,\theta)|\ls\Big|\int_0^\infty\int\chi_+^{-3/2}(s-\sigma-\tau)G(\tau,\theta,s,0)dsd\tau\Big|\\
&\ls\int_0^\infty\w{\sigma+\tau}^{-3/2}\w{\tau}^{-1}d\tau\\
&\ls\int_{2|\sigma|}^\infty\w{\tau}^{-5/2}d\tau+\w{\sigma}^{-1}\int_{|\sigma|/2}^{2|\sigma|}\w{\sigma+\tau}^{-3/2}d\tau
+\w{\sigma}^{-3/2}\int_0^{|\sigma|/2}\w{\tau}^{-1}d\tau\\
&\ls\w{\sigma}^{-1}.
\end{split}
\end{equation*}
Since the proofs on the other cases of $i, j$ for \eqref{FRF1-decay} are analogous, we omit them here.

Finally, we prove \eqref{LWE-err1} only for $a=0$ because the other cases can be analogously obtained.
By \eqref{LWE-err-pf4} and \eqref{LWE-err-pf5}, one has
\begin{equation}\label{LWE-err-pf7}
\begin{split}
&|r^{1/2}w^c_1(t,x)-F_0^1(\sigma,\theta)|
\ls\Big|\int_t^\infty\int\chi_+^{-1/2}(s-\sigma-\tau)G(\tau,\theta,s,0)dsd\tau\Big|\\
&+\Big|\int_0^t\int[\chi_+^{-1/2}(s-\sigma-\tau+\frac{z(\sigma+\tau)^2}{2})G(\tau,\theta,s,z)
-\chi_+^{-1/2}(s-\sigma-\tau)G(\tau,\theta,s,0)]dsd\tau\Big|\\
&=I^1_1+I^1_2.
\end{split}
\end{equation}
It follows from the homogeneity of $\chi_+^{-1/2}$ and \eqref{loc:decay} that
\begin{equation*}
I^1_1\ls\int_t^\infty\w{\sigma+\tau}^{-1/2}\w{\tau}^{-1}d\tau.
\end{equation*}
For $t\ge2|\sigma|$, we can see that $\tau\ge2|\sigma|$ and
\begin{equation*}
I^1_1\ls\int_t^\infty\w{\tau}^{-3/2}d\tau\ls\w{t}^{-1/2}.
\end{equation*}
For $t\le2|\sigma|$, one has
\begin{equation*}
\begin{split}
I^1_1&\ls\w{t}^{-1}\int_t^{2|\sigma|}\w{\sigma+\tau}^{-1/2}d\tau
+\int_{2|\sigma|}^\infty\w{\tau}^{-3/2}d\tau\\
&\ls\w{t}^{-1}\w{\sigma}^{1/2}+\w{\sigma}^{-1/2}\ls\w{t}^{-1/2}.
\end{split}
\end{equation*}
Then it holds that by \eqref{LWE-err-pf7}
\begin{equation}\label{LWE-err-pf8}
\begin{split}
&|r^{1/2}w^c_1(t,x)-F_0^1(\sigma,\theta)|\ls\w{t}^{-1/2}\\
&\qquad+\Big|\int_0^{1/r}\int_0^t\int\frac{\p}{\p z}[\chi_+^{-1/2}(s-\sigma-\tau+\frac{z(\sigma+\tau)^2}{2})G(\tau,\theta,s,z)]dsd\tau dz\Big|\\
&\ls\w{t}^{-1/2}+r^{-1}\sup_{0<z\le1/r}\Big|\int_0^t\int\frac{\p}{\p z}[\chi_+^{-1/2}(s-\sigma-\tau+\frac{z(\sigma+\tau)^2}{2})G(\tau,\theta,s,z)]dsd\tau\Big|\\
&\ls\w{t}^{-1/2}+r^{-1}(I^1_{21}+I^1_{22}),
\end{split}
\end{equation}
where
\begin{equation}\label{LWE-err-pf9}
\begin{split}
I^1_{21}&=\sup_{0<z\le1/r}\Big|\int_0^t\int(\sigma+\tau)^2\chi_+^{-3/2}(s-\sigma-\tau+\frac{z(\sigma+\tau)^2}{2})G(\tau,\theta,s,z)d\tau ds\Big|,\\
I^1_{22}&=\sup_{0<z\le1/r}\Big|\int_0^t\int\chi_+^{-1/2}(s-\sigma-\tau+\frac{z(\sigma+\tau)^2}{2})G_z'(\tau,\theta,s,z)d\tau ds\Big|.
\end{split}
\end{equation}
By the homogeneity of $\chi_+^{-3/2}$ and $|\frac{z(\sigma+\tau)^2}{2}-\sigma-\tau|=\frac{z(2/z-r+t-\tau)}{2}|\sigma+\tau|\ge\frac{|\sigma+\tau|}{2}$ with $z\in(0,1/r]$, we arrive at
\begin{equation}\label{LWE-err-pf10}
I^1_{21}\ls\int_0^t\sqrt{|\sigma+\tau|}\w{\tau}^{-1}d\tau=\int_0^t\sqrt{|r-t+\tau|}\w{\tau}^{-1}d\tau\ls\sqrt{r}\ln(2+t),
\end{equation}
where \eqref{loc:decay} has been used.
Analogously, one can achieve
\begin{equation}\label{LWE-err-pf11}
I^1_{22}\ls\int_0^t\w{\sigma+\tau}^{-1/2}\w{\tau}^{-1}d\tau\ls\w{\sigma}^{-1/2}\ln(2+t+|\sigma|).
\end{equation}
Collecting \eqref{LWE-err-pf8}-\eqref{LWE-err-pf11} yields the first inequality in \eqref{LWE-err1} with $a=0$.
Next, we prove the second inequality in \eqref{LWE-err1} with $a=0$.
Analogously to \eqref{LWE-err-pf7}, we have
\begin{equation}\label{LWE-err-pf12}
\begin{split}
&|\p_t(r^{1/2}w^c_1(t,x)-F_0^1(\sigma,\theta))|\ls I^2_1+I^2_2+I^2_3\quad\text{with}\\
&I^2_1=\Big|\int\chi_+^{-1/2}(s-\sigma-t+\frac{z(\sigma+t)^2}{2})G(t,\theta,s,z)ds\Big|\\
&\qquad+\Big|\int\chi_+^{-1/2}(s-\sigma-t)G(t,\theta,s,0)ds\Big|,\\
&I^2_2=\Big|\int_t^\infty\int\chi_+^{-3/2}(s-\sigma-\tau)G(\tau,\theta,s,0)dsd\tau\Big|,\\
&I^2_3=\Big|\int_0^t\int z(\sigma+\tau)\chi_+^{-3/2}(s-\sigma-\tau+\frac{z(\sigma+\tau)^2}{2})G(\tau,\theta,s,z)dsd\tau\Big|\\
&\qquad+\Big|\int_0^t\int[\chi_+^{-3/2}(s-\sigma-\tau+\frac{z(\sigma+\tau)^2}{2})G(\tau,\theta,s,z)\\
&\qquad\quad-\chi_+^{-3/2}(s-\sigma-\tau)G(\tau,\theta,s,0)]dsd\tau\Big|.
\end{split}
\end{equation}
Similarly to \eqref{LWE-err-pf8}-\eqref{LWE-err-pf11}, one arrives at
\begin{equation}\label{LWE-err-pf13}
\begin{split}
&I^2_1\ls\Big|\int\chi_+^{-1/2}(s-r/2)G(t,\theta,s,z)ds\Big|+\Big|\int\chi_+^{-1/2}(s-r)G(t,\theta,s,0)ds\Big|\ls\w{t}^{-1}\w{r}^{-1/2},\\
&I^2_2\ls\int_t^\infty\w{\sigma+\tau}^{-3/2}\w{\tau}^{-1}d\tau\ls\w{t}^{-1}\int_t^\infty\w{r-t+\tau}^{-3/2}d\tau\ls\w{t}^{-1}\w{r}^{-1/2}
\end{split}
\end{equation}
and
\begin{equation}\label{LWE-err-pf14}
\begin{split}
I^2_3&\ls r^{-1}\sup_{0<z\le 1/r}\Big|\int_0^t\int\w{\sigma+\tau}^2\chi_+^{-5/2}(s-\sigma-\tau+\frac{z(\sigma+\tau)^2}{2})G(\tau,\theta,s,z)dsd\tau\Big|\\
&\quad+r^{-1}\int_0^t\w{\sigma+\tau}^{-1/2}\w{\tau}^{-1}d\tau\\
&\ls r^{-1}\int_0^t\w{\sigma+\tau}^{-1/2}\w{\tau}^{-1}d\tau\ls r^{-1}\w{\sigma}^{-1/2}\ln(2+t+|\sigma|).
\end{split}
\end{equation}
Due to
\begin{equation*}
\p_1=\frac{x_1}{r}\p_r-\frac{x_2}{r^2}\p_{\theta},\quad \p_2=\frac{x_2}{r}\p_r+\frac{x_1}{r^2}\p_{\theta}\quad\text{and}\quad\Omega=\p_{\theta},
\end{equation*}
the estimate of $\p_x(r^{1/2}w^c_1(t,x)-F_0^1(\sigma,\theta))$ can be reduced to the treatment on $\p_r(r^{1/2}w^c_1(t,x)-F_0^1(\sigma,\theta))$.
In this case,  one needs to deal with such an extra term
\begin{equation}\label{LWE-err-pf15}
\begin{split}
I^2_4&\ls r^{-2}\Big|\int_0^t\int\w{\sigma+\tau}^2\chi_+^{-3/2}(s-\sigma-\tau+\frac{z(\sigma+\tau)^2}{2})G(\tau,\theta,s,z)dsd\tau\Big|\\
&\ls r^{-2}\int_0^t\w{\sigma+\tau}^{1/2}\w{\tau}^{-1}d\tau\ls r^{-3/2}\ln(2+t).
\end{split}
\end{equation}
Substituting \eqref{LWE-err-pf13}-\eqref{LWE-err-pf15} into \eqref{LWE-err-pf12} yields the second inequality in \eqref{LWE-err1} with $a=0$.
\end{proof}

\begin{proof}[Proof of Proposition \ref{prop31}]
According to the definitions \eqref{prop31:pf1} and \eqref{prop31:pf2}, we have
$w_{B}(t,x)=w^c=w^c_1+w^c_2$ in the region $|x|\ge1$.
Note that $w^c_2$ can be treated more easily than $w^c_1$,
then there exists a smooth function $F_0^2(\sigma,\theta)$ such that
$w_2^c-r^{-\f12}F_0^2(\sigma,\theta)$ and $F_0^2(\sigma,\theta)$ admit
the analogous estimates \eqref{LWE-err1} and \eqref{FRF1-decay},  respectively.
Let $F_0(\sigma,\theta)=F_0^1(\sigma,\theta)+F_0^2(\sigma,\theta)$.
Then the proof of Proposition \ref{prop31} can be completed.
\end{proof}

At last, we give some estimates of the solution  $w_{B}$ to \eqref{LWE}.
\begin{lemma}
Let $w_{B}$ be the solution to the homogeneous linear IBVP \eqref{LWE}. Then one has
\begin{equation}\label{LWE-decay}
\begin{split}
|\p Z^aw_{B}(t,x)|&\ls\w{x}^{-1/2}\w{t-|x|}^{-1},\\
|Z^aw_{B}(t,x)|&\ls\w{t+|x|}^{-1/2}\w{t-|x|}^{-1/2}\ln^2(2+t+|x|).
\end{split}
\end{equation}
\end{lemma}
\begin{proof}
The estimates \eqref{LWE-decay} in the region $x\in\cK\cap\{x: |x|\le2\}$ can be derived from \eqref{loc:decay} with $H=0$ and the Sobolev embedding.
To obtain the estimates in the region $x\in\cK\cap\{x: |x|\ge2\}$, let $w_{B}^a=\chi_{[1,2]}(x)Z^aw_{B}$ be a function on $\R^2$, which verifies the Cauchy problem $\Box w_{B}^a=[-\Delta,\chi_{[1,2]}]Z^aw_{B}$.
Since $\Box w_{B}^a$ is supported in $|x|\le2$, applying \eqref{pw:compact} and \eqref{dpw:compact} to $\Box w_{B}^a$ with \eqref{loc:decay} implies that
\begin{equation*}
\begin{split}
|\p(\chi_{[1,2]}(x)Z^aw_{B})|&\ls\w{x}^{-1/2}\w{t-|x|}^{-1},\\
|\chi_{[1,2]}(x)Z^aw_{B}|&\ls\w{t+|x|}^{-1/2}\w{t-|x|}^{-1/2}\ln^2(2+t+|x|).
\end{split}
\end{equation*}
This, together with the facts of $\chi_{[1,2]}(x)Z^aw_{B}=Z^aw_{B}$ and $\p(\chi_{[1,2]}(x)Z^aw_{B})=\p Z^aw_{B}$ in the region $x\in\cK\cap\{x: |x|\ge2\}$, yields \eqref{LWE-decay}.
\end{proof}

\section{Approximate solutions}\label{sect4}

\subsection{Construction of the approximate solution}

Let $U(\tau,\sigma,\theta)$ be the solution to
\begin{equation}\label{profile-eqn}
\left\{
\begin{aligned}
&\p^2_{\tau\sigma}U-Q_{1st}(\tilde\omega)\p_{\sigma}U\p^2_{\sigma}U=0,\\
&U(0,\sigma,\theta)=F_0(\sigma,\theta),\\
&\supp_{\sigma} U\subset \{\sigma: \sigma\le M_0\},
\end{aligned}
\right.
\end{equation}
where $Q_{1st}(\tilde\omega)$ and $F_0$ are given by \eqref{nonlinear} and \eqref{def:F0}, respectively.
Set $x=r(\cos\theta,\sin\theta)$, $\sigma=r-t$, $\tau=\ve\sqrt{1+t}$ and
\begin{equation}\label{app-solution}
u_{ap}(t,x)=\ve\chi(\ve t)w_{B}(t,x)+\ve(1-\chi(\ve t))\chi(-3\ve\sigma)r^{-1/2}U(\tau,\sigma,\theta),
\end{equation}
where $\chi=1-\chi_{[1,2]}$, $w_{B}$ is the solution to the homogeneous linear IBVP \eqref{LWE}.
From \eqref{LWE}, $\cB u_{ap}|_{[0,\infty)\times \p\cK}=0$ obviously holds.

\begin{proposition}
The initial value problem \eqref{profile-eqn} admits a unique smooth solution $U \in C^{\infty}([0, \tau_0)\times\mathbb{R}\times[0, 2\pi])$, 
where $\tau_0$ is given by \eqref{thm1:lifespan}.
\end{proposition}

\begin{remark}\label{rmk:tau}
Note that \(0<\tau_0<+\infty\) holds whenever \((u_0^1, u_1^1)\not\equiv 0\).
Indeed, define
\begin{equation*}
\cF(\sigma, \theta)=Q_{1st}(\tilde\omega)F_0(\sigma, \theta).
\end{equation*}
Due to $Q_{1st}(\tilde\omega)\not\equiv 0$, it follows from Proposition \ref{prop31} that
\[
\cF\not\equiv 0,\quad \supp_{\sigma}\cF\subset\{\sigma: \sigma\leq M_0\}\quad\text{and}\quad
\lim_{\sigma\rightarrow-\infty}\cF(\sigma, \theta)=0.
\]
Consequently, there exists a point \((\sigma_0, \theta_0)\in\mathbb R\times [0,2\pi]\) 
such that \(\partial_\sigma^2\mathcal F(\sigma_0, \theta_0)>0\).
This implies
\begin{equation*}
Q_{1st}(\tilde\omega_0)\partial_\sigma^2F_0(\sigma_0, \theta_0)>0
\end{equation*}
with \(\tilde\omega_0=(-1, \cos\theta_0, \sin\theta_0)\).
\end{remark}

\begin{proof}
Let $V(\tau, \sigma, \theta)=\partial_\sigma U(\tau, \sigma, \theta)$. Then we have that from \eqref{profile-eqn},
\begin{equation}\label{def:tau:pf1}
\begin{cases}
\ds \p_{\tau}V- Q_{1st}(\tilde\omega)V\p_\sigma V=0,\quad\tau>0, \ \sigma\in\R,\ \theta\in[0, 2\pi],\\
V(0,\sigma,\theta)=\partial_\sigma F_0(\sigma,\theta).
\end{cases}
\end{equation}
The characteristics $\sigma=\sigma(s, \tau, \theta)$ of \eqref{def:tau:pf1} issuing from $(s, 0, \theta)$ is given by
\begin{equation*}
\begin{cases}
\ds \frac{\partial \sigma}{\partial \tau}(s, \tau, \theta)
=-Q_{1st}(\tilde\omega) V(\tau, \sigma(s, \tau, \theta), \theta),\\
\sigma(s, 0, \theta)=s.
\end{cases}
\end{equation*}
Along such a characteristics, it follows from \eqref{def:tau:pf1} that
\[
V(\tau, \sigma(s, \tau, \theta), \theta)=\partial_\sigma F_0(s, \theta)=\partial_\sigma U(\tau, \sigma(s, \tau, \theta), \theta),
\]
which implies
$\sigma(s, \tau, \theta)=s-Q_{1st}(\tilde\omega)\partial_\sigma F_0(s,\theta)\tau$
and $\frac{\partial\sigma}{\partial s}(s, \tau, \theta)=1-Q_{1st}(\tilde\omega)\partial_\sigma^2 F_0(s,\theta)\tau$.
By the implicit function theorem, there exists a smooth function $s=s(\tau, \sigma, \theta)$ for $\tau<\tau_0$ such that
$\partial_\sigma U(\tau, \sigma, \theta)=\partial_\sigma F_0(s(\tau, \sigma, \theta), \theta)$.

Since $U(\tau, \sigma, \theta)\equiv 0$ holds when $\sigma>M_0$, we conclude that for $\tau<\tau_0$,
\begin{equation*}
U(\tau, \sigma, \theta)=\int_{M_0}^\sigma\partial_\sigma F_0(s(\tau, \tilde\sigma, \theta),\theta)d\tilde\sigma.
\end{equation*}
\end{proof}

\begin{corollary}
    For any positive constant $B<\tau_0$, if $t\ge\frac{1}{\ve}$ and $r\ge t/3$, then
\begin{equation}\label{Uestimate}
\begin{aligned}
    &|Z^a\partial_\sigma U(\tau, \sigma, \theta)|\le C_{B}\w{\sigma}^{-1},\quad
    |Z^a\partial_{\sigma\tau}^2 U(\tau, \sigma, \theta)|\le C_B\w{\sigma}^{-2}, \\[4pt]
    &|Z^aU(\tau, \sigma, \theta)|\le C_{B}\ln(2+|\sigma|),\quad
    |Z^a\partial_\tau U|\le C_{B},
\end{aligned}
\end{equation}
    where $C_B$ is a generic positive constant depending only on $B$.
\end{corollary}

\subsection{Estimates of the approximate solution}

Let $u$ be the solution of \eqref{QWE} and $u_{ap}$ be defined by \eqref{app-solution}.
Denote the error solution
\begin{equation}\label{error-solution}
v(t,x)=u(t,x)-u_{ap}(t,x).
\end{equation}
Then $v$ satisfies $\cB v|_{[0,\infty)\times\p\cK}=0$, $v(0,x)=O(\ve^2)$ and
\begin{equation}\label{error-eqn}
\begin{split}
&\Box v=-\sum_{\alpha,\beta,\mu=0}^2Q^{\alpha\beta\mu}\Big\{\p^2_{\alpha\beta}v(\p_{\mu}v+\p_{\mu}u_{ap})
+\p^2_{\alpha\beta}u_{ap}\p_{\mu}v\Big\}+R_{ap},\\
&R_{ap}=-\Box u_{ap}-\sum_{\alpha,\beta,\mu=0}^2Q^{\alpha\beta\mu}\p^2_{\alpha\beta}u_{ap}\p_{\mu}u_{ap}.
\end{split}
\end{equation}

\begin{lemma}
Under the assumptions of Theorem \ref{thm1}, for any positive constant $B<\tau_0$ and $t\le\frac{B^2}{\ve^2}-1$, one has
\begin{equation}\label{app-pw}
\begin{split}
|Z^au_{ap}(t,x)|&\ls\ve\w{t+|x|}^{\tilde\ve-1/2}\w{t-|x|}^{-1/2},\quad t\le\frac{2}{\ve},\\
|\p Z^au_{ap}(t,x)|&\ls\ve\w{x}^{-1/2}\w{t-|x|}^{-1},
\end{split}
\end{equation}
where and below $\tilde\ve=10^{-3}$.
\end{lemma}
\begin{proof}
When $t\le\frac{1}{\ve}$, $u_{ap}=\ve w_{B}$ holds and \eqref{app-pw} follows from \eqref{LWE-decay}.

When $\frac{1}{\ve}\le t\le\frac{2}{\ve}$, one has
\begin{equation*}
\begin{split}
u_{ap}&=\ve w_{B}(\chi(\ve t)+(1-\chi(\ve t))\chi(-3\ve\sigma))\\
&\quad+\ve(1-\chi(\ve t))\chi(-3\ve\sigma)(r^{-1/2}F_0-w_{B})\\
&\quad+\ve(1-\chi(\ve t))\chi(-3\ve\sigma)r^{-1/2}(U-F_0).
\end{split}
\end{equation*}
For the last line above, \eqref{Uestimate} implies $|U-F_0|=|\int_0^{\tau}\p_{\tau}Uds|\ls\tau\ls\sqrt{\ve}$.
Therefore, \eqref{app-pw} can be achieved by \eqref{LWE-error} and \eqref{LWE-decay}.

When $t\ge\frac{2}{\ve}$, $u_{ap}=\ve \chi(-3\ve\sigma)r^{-1/2}U(\tau,\sigma,\theta)$ holds and \eqref{app-pw} can
be obtained from \eqref{Uestimate}.
\end{proof}

\begin{lemma}\label{Lem4.4}
Under the assumptions of Theorem \ref{thm1},
for any positive constant $B<\tau_0$ and $t\le\frac{B^2}{\ve^2}-1$, one has
\begin{equation}\label{Rapp-L2}
\begin{split}
\|Z^aR_{ap}(t)\|_{L^2(\cK)}&\le C_B\ve^2\w{t}^{\tilde\ve-1/2},\\
\int_0^{B^2/\ve^2-1}\|Z^aR_{ap}(t)\|_{L^2(\cK)}dt&\le C_B\ve^{3/2-\tilde\ve}.
\end{split}
\end{equation}
\end{lemma}
\begin{proof}
The proof will be separated into three parts: $\ve t\le1$, $1\le\ve t\le2$ and $\ve t\ge2$.

When $\ve t\le1$, one has  $u_{ap}=\ve w_{B}$ and $\Box u_{ap}=\ve\Box w_{B}\equiv0$.
On the other hand, \eqref{LWE-decay} gives
\begin{equation}\label{Rapp-L2-pf1}
\begin{split}
|Z^aR_{ap}|&\ls\ve^2\w{t+|x|}^{-1}\w{t-|x|}^{-1},\\
\|Z^aR_{ap}\|_{L^2(\cK)}&\ls\ve^2\w{t}^{-1}\|\w{t-|x|}^{-1}\|_{L^2(|x|\le t+M_0)}\ls\ve^2\w{t}^{-1/2},
\end{split}
\end{equation}
which yields the first line in \eqref{Rapp-L2} for $\ve t\le1$.
Then integrating \eqref{Rapp-L2-pf1} over $[0,1/\ve]$ implies
\begin{equation}\label{Rapp-L2-pf2}
\int_0^{1/\ve}\|Z^aR_{ap}(t)\|_{L^2(\cK)}dt\ls\ve^2\int_0^{1/\ve}\w{t}^{-1/2}dt\ls\ve^{3/2}.
\end{equation}

When $1\le\ve t\le2$, by the definition \eqref{app-solution} we have
\begin{equation}\label{Rapp-L2-pf3}
\begin{split}
\Box u_{ap}&=\ve\Box[(1-\chi(\ve t))\chi(-3\ve\sigma)r^{-1/2}(U(\tau,\sigma,\theta)-F_0(\sigma,\theta))]\\
&\quad+\ve\Box[(1-\chi(\ve t))\chi(-3\ve\sigma)(r^{-1/2}F_0(\sigma,\theta)-w_{B}(t,x))]\\
&\quad-\ve\Box[(1-\chi(\ve t))(1-\chi(-3\ve\sigma))w_{B}(t,x)]\\
&=R_{ap,1}+R_{ap,2}+R_{ap,3}.
\end{split}
\end{equation}
By $\Box=\p_t^2-\p_r^2-\p_r/r-\p^2_{\theta}/r^2$, one has $R_{ap,1}=R_{ap,11}+R_{ap,12}+R_{ap,13}+R_{ap,14}$, where
\begin{equation}\label{Rapp-L2-pf4}
\begin{split}
&R_{ap,11}=\ve\Box[(1-\chi(\ve t))\chi(-3\ve\sigma)]r^{-1/2}(U(\tau,\sigma,\theta)-F_0(\sigma,\theta)),\\
&R_{ap,12}=\ve\p[(1-\chi(\ve t))\chi(-3\ve\sigma)]\p[r^{-1/2}(U(\tau,\sigma,\theta)-F_0(\sigma,\theta))],\\
&R_{ap,13}=-\ve(1-\chi(\ve t))\chi(-3\ve\sigma)r^{-5/2}(\p^2_{\theta}+1/4)(U(\tau,\sigma,\theta)-F_0(\sigma,\theta)),\\
&R_{ap,14}=\ve(1-\chi(\ve t))\chi(-3\ve\sigma)r^{-1/2}(\p_t^2-\p_r^2)(U(\tau,\sigma,\theta)-F_0(\sigma,\theta)).
\end{split}
\end{equation}
It follows from $U(\tau,\sigma,\theta)-F_0(\sigma,\theta)=\int_0^{\tau}\p_{\tau}U(s,\sigma,\theta)ds$ and \eqref{Uestimate} that
\begin{equation}\label{Rapp-L2-pf5}
\begin{split}
|Z^aR_{ap,11}|+|Z^aR_{ap,12}|+|Z^aR_{ap,13}|&\ls\ve^2\ve\sqrt{1+t}\w{t}^{-1/2}\w{t-|x|}^{\tilde\ve-1/2},\\
\|Z^aR_{ap,11}\|_{L^2}+\|Z^aR_{ap,12}\|_{L^2}+\|Z^aR_{ap,13}\|_{L^2}&\ls\ve^3\w{t}^{1/2+\tilde\ve}
\end{split}
\end{equation}
and
\begin{equation}\label{Rapp-L2-pf5}
\begin{split}
|(\p_t^2-\p_r^2)(U(\tau,\sigma,\theta)-F_0(\sigma,\theta))|
&=|(\p_t-\p_r)(\p_t+\p_r)\int_0^{\tau}\p_{\tau}U(s,\sigma,\theta)ds|\\
&=|(\p_t-\p_r)\frac{\ve}{2\sqrt{1+t}}\p_{\tau}U(\tau,\sigma,\theta)|\\
&\ls\ve\w{t}^{-1/2}\w{t-|x|}^{-2}+\ve\w{t}^{-3/2},\\
\|Z^aR_{ap,14}\|_{L^2}&\ls\ve^2\w{t}^{-1/2}.
\end{split}
\end{equation}
Collecting \eqref{Rapp-L2-pf3}-\eqref{Rapp-L2-pf5} gives
\begin{equation}\label{Rapp-L2-pf6}
\begin{split}
\|Z^aR_{ap,1}(t)\|_{L^2(\cK)}&\ls\ve^2\w{t}^{\tilde\ve-1/2},\\
\int_{1/\ve}^{2/\ve}\|Z^aR_{ap,1}(t)\|_{L^2(\cK)}dt&\ls\ve^{3/2-\tilde\ve}.
\end{split}
\end{equation}
The estimate of $R_{ap,2}$ can be obtained from
\begin{equation}\label{Rapp-L2-pf7}
\begin{split}
R_{ap,2}&=\ve\Box[(1-\chi(\ve t))\chi(-3\ve\sigma)](r^{-1/2}F_0(\sigma,\theta)-w_{B}(t,x))\\
&\quad+\ve\p[(1-\chi(\ve t))\chi(-3\ve\sigma)]\p(r^{-1/2}F_0(\sigma,\theta)-w_{B}(t,x))\\
&\quad-\ve(1-\chi(\ve t))\chi(-3\ve\sigma)r^{-5/2}(\p^2_{\theta}+1/4)(F_0(\sigma,\theta)-\sqrt{r}w_{B}(t,x))\\
&\quad+\ve(1-\chi(\ve t))\chi(-3\ve\sigma)r^{-1/2}(\p_t^2-\p_r^2)(F_0(\sigma,\theta)-\sqrt{r}w_{B}(t,x)).
\end{split}
\end{equation}
For the last line, from $(\p_t+\p_r)F_0(\sigma,\theta)\equiv0$ we arrive at
\begin{equation}\label{Rapp-L2-pf8}
|(\p_t^2-\p_r^2)(\sqrt{r}w_{B})|\ls\w{r}^{-3/2}|(\p^2_{\theta}+1/4)w_{B}|+\sqrt{r}|\Box w_{B}|\ls\w{r}^{-3/2}|(\p^2_{\theta}+1/4)w_{B}|.
\end{equation}
Applying \eqref{LWE-error} and \eqref{LWE-decay} to $R_{ap,3}$ in \eqref{Rapp-L2-pf3}, \eqref{Rapp-L2-pf7} and \eqref{Rapp-L2-pf8} yields
\begin{equation*}
\begin{split}
|Z^aR_{ap,2}|+|Z^aR_{ap,3}|&\ls\ve^3\w{t}^{\tilde\ve-1}+\ve^2\w{t}^{\tilde\ve-3/2}\w{t-|x|}^{-1/2}+\ve\w{t}^{\tilde\ve-5/2}\w{t-|x|}^{-1/2}\\
&\quad+\tilde\ve\w{t}^{-1}\w{x}^{-1/2},
\end{split}
\end{equation*}
which leads to
\begin{equation}\label{Rapp-L2-pf9}
\|Z^aR_{ap,2}\|_{L^2}+\|Z^aR_{ap,3}\|_{L^2}\ls\ve^2\w{t}^{-1/2}.
\end{equation}
For the remaining quadratic terms in \eqref{error-eqn}, the estimates can be analogously obtained as in \eqref{Rapp-L2-pf1} 
by utilizing \eqref{app-pw}.
Collecting \eqref{Rapp-L2-pf2}, \eqref{Rapp-L2-pf3}, \eqref{Rapp-L2-pf6} and \eqref{Rapp-L2-pf9} implies \eqref{Rapp-L2} for $\ve t\le2$.

Finally, we focus on the period of $\ds\frac{2}{\ve}\le t\le\frac{B^2}{\ve^2}-1$. In this case, the approximate solution is
\[
u_{ap}=\ve r^{-1/2}\chi(-3\ve\sigma)U(\tau,\sigma,\theta).
\]
By the support property of the cutoff function $\chi$, it suffices to consider \eqref{Rapp-L2} for $r\ge\frac{2}{3}t$.
It follows from the expression of $u_{ap}$, the estimates in \eqref{Uestimate} and  the equation \eqref{profile-eqn} that
\begin{equation*}
\begin{split}
R_{ap}=&\chi(-3\ve\sigma)\Bigl\{\frac{\ve^2}{r^{1/2}t^{1/2}}\p_{\sigma\tau}^2U-\chi(-3\ve\sigma)
Q_{1st}(\tilde\omega)\frac{\ve^2}{r}\p_\sigma U\p_\sigma^2 U\Bigr\}\\
&+O\bigl(\ve^3t^{-1}\ln\frac{1}{\ve}\bigr)+O\bigl(\ve t^{-5/2}\ln(2+|\sigma|)\bigr)\\
=&\chi(-3\ve\sigma)\ve^2\p_{\sigma\tau}^2U\Bigl\{\frac{r-t}{rt^{1/2}(r^{1/2}+t^{1/2})}+\frac{1}{r}\bigl(1-\chi(-3\ve\sigma)\bigr)\Bigr\}\\
&+O\bigl(\ve^3t^{-1}\ln\frac{1}{\ve}\bigr)+O\bigl(\ve t^{-5/2}\ln(2+|\sigma|)\bigr)\\
=&O\bigl(\ve^2t^{-2}\w\sigma^{-1}\bigr)+O\bigl(\ve^3t^{-1}\ln\frac{1}{\ve}\bigr)+O\bigl(\ve t^{-5/2}\ln(2+|\sigma|)\bigr).
\end{split}
\end{equation*}
Note that $R_{ap}$ is supported in $\{-\frac{2}{3\ve}\le\sigma\le M_0,r\le t+M_0\}$.
Then, we can achieve
\[
\|Z^aR_{ap}(t)\|_{L^2}\le C_B\bigl(\ve^2t^{-3/2}+\ve^{5/2}t^{-1/2}\ln\frac{1}{\ve}+\ve^{1/2}t^{-2}\ln\frac{1}{\ve}\bigr).
\]
Therefore, integrating the above inequality over $[2/\ve,B^2/\ve^2-1]$ yields
\begin{equation}
	\int_{2/\ve}^{B^2/\ve^2-1}\|Z^aR_{ap}\|_{L^2}dt\le C_B\ve^{3/2}\ln\frac{1}{\ve}.
\end{equation}
This completes the proof of Lemma \ref{Lem4.4}.
\end{proof}

\section{Proof of Theorem \ref{thm1}}\label{sect5}

\subsection{Bootstrap assumption}

We make the following bootstrap assumption on the error solution $v$ defined in \eqref{error-solution} for $t\in[0,B^2/\ve^2-1]$ with any
fixed positive constant $B<\tau_0$,
\begin{equation}\label{BA}
\begin{split}
\sum_{|a|\le N}|\p Z^av|&\le\ve^{1.4}\w{x}^{-1/2},\\
\sum_{|a|\le N}|\p Z^av|&\le\ve_1\w{x}^{-1/2}\w{t-|x|}^{-1}\w{t}^{2\tilde\ve},
\end{split}
\end{equation}
where $\ve_1=C_0\ve$ with $C_0\ge1$ being determined later.
Choosing $\ve_0>0$ small enough such that
\begin{equation}\label{BA1}
e^{C_0}\ve_0^{10^{-3}}\le1.
\end{equation}
Note that due to $\supp(u_0,u_1)\subset\{x: |x|\le M_0\}$,
the solution $u$ of problem \eqref{QWE} is supported in $\{x\in\cK:|x|\le t+M_0\}$.

\subsection{Energy estimates}

\begin{lemma}\label{lem:energy:time}
Under the conditions of Theorem \ref{thm1} and assumption \eqref{BA}, it holds that for $t\in[0,\f{B^2}{\ve^2}-1]$,
\begin{equation}\label{energy:time}
\sup_{s\in[0,t]}\sum_{i\le2N}\|\p\p_t^iv(s)\|_{L^2(\cK)}
\ls\ve^{3/2-\tilde\ve}+\ve_1\sum_{j\le2N-1}\int_0^t\|\p^{\le1}\p\p_t^jv(s)\|_{L^2(\cK)}\frac{ds}{\w{s}^{1/2}}.
\end{equation}
\end{lemma}
\begin{proof}
At first, applying $\p_t^i$ with $i=0,1,\cdots,2N$ to \eqref{error-eqn} yields
\begin{equation}\label{energy:time1}
\Box\p_t^iv=-\sum_{j+k=i}\sum_{\alpha,\beta,\mu=0}^2C^i_{jk}Q^{\alpha\beta\mu}
\Big\{\p^2_{\alpha\beta}\p_t^jv(\p_{\mu}\p_t^kv+\p_{\mu}\p_t^ku_{ap})
+\p^2_{\alpha\beta}\p_t^ju_{ap}\p_{\mu}\p_t^kv\Big\}+\p_t^iR_{ap},
\end{equation}
where $C^i_{jk}=\frac{i!}{j!k!}$.
Multiplying \eqref{energy:time1} by $\p_t\p_t^iv$ leads to
\begin{equation}\label{energy:time2}
\begin{split}
&\quad\;\frac12\p_t(|\p_t\p_t^iv|^2+|\nabla\p_t^iv|^2)-\dive(\p_t\p_t^iv\nabla\p_t^iv)\\
&=-\sum_{\alpha,\beta,\mu=0}^2Q^{\alpha\beta\mu}\p_t\p_t^iv\p^2_{\alpha\beta}\p_t^iv(\p_{\mu}v+\p_{\mu}u_{ap})
+\p_t\p_t^iv\p_t^iR_{ap}+J_1^i
\end{split}
\end{equation}
and
\begin{equation}\label{energy:time3}
\begin{split}
J_1^i&=-\sum_{\substack{j+k=i,\\j<i}}\sum_{\alpha,\beta,\mu=0}^2C^i_{jk}Q^{\alpha\beta\mu}\p_t\p_t^iv\p^2_{\alpha\beta}\p_t^jv(\p_{\mu}\p_t^kv
+\p_{\mu}\p_t^ku_{ap})\\
&\qquad-\sum_{j+k=i}\sum_{\alpha,\beta,\mu=0}^2C^i_{jk}Q^{\alpha\beta\mu}\p_t\p_t^iv\p^2_{\alpha\beta}\p_t^ju_{ap}\p_{\mu}\p_t^kv.
\end{split}
\end{equation}
For the first term in the second line of \eqref{energy:time2}, it follows from a direct computation that
\begin{equation}\label{energy:time4}
\begin{split}
&\quad\;-Q^{\alpha\beta\mu}\p_t\p_t^iv\p^2_{\alpha\beta}\p_t^iv(\p_{\mu}v+\p_{\mu}u_{ap})\\
&=-\p_{\alpha}[Q^{\alpha\beta\mu}\p_t\p_t^iv\p_{\beta}\p_t^iv(\p_{\mu}v+\p_{\mu}u_{ap})]
+Q^{\alpha\beta\mu}\p_t\p_t^iv\p_{\beta}\p_t^iv(\p^2_{\alpha\mu}v+\p^2_{\alpha\mu}u_{ap})\\
&\quad+\frac12\p_t[Q^{\alpha\beta\mu}\p_{\alpha}\p_t^iv\p_{\beta}\p_t^iv(\p_{\mu}v+\p_{\mu}u_{ap})]
-\frac12Q^{\alpha\beta\mu}\p_{\alpha}\p_t^iv\p_{\beta}\p_t^iv(\p^2_{0\mu}v+\p^2_{0\mu}u_{ap}),
\end{split}
\end{equation}
where the summation $\ds\sum_{\alpha,\beta,\mu=0}^2$ is omitted.
Integrating \eqref{energy:time2}-\eqref{energy:time4} over $[0,t]\times\cK$
with the boundary conditions $\cB\p_t^lv|_{[0,\infty)\times\p\cK}=\cB\p_t^lu_{ap}|_{[0,\infty)\times\p\cK}=0$ for any integer $l\ge0$
and using the admissible condition \eqref{admiss:condition} yield
\begin{equation}\label{energy:time5}
\|\p\p_t^iv(t)\|^2_{L^2(\cK)}\ls\|\p\p_t^iv(0)\|^2_{L^2(\cK)}+\ve_1\|\p\p_t^iv(t)\|^2_{L^2(\cK)}
+\int_0^t\|(J_1^i,J_2^i,\p_t\p_t^iv\p_t^iR_{ap})\|_{L^1(\cK)}ds,
\end{equation}
where \eqref{app-pw} and \eqref{BA} have been used, and $J_2^i$ is defined by
\begin{equation}\label{energy:time6}
J_2^i=Q^{\alpha\beta\mu}\p_t\p_t^iv\p_{\beta}\p_t^iv(\p^2_{\alpha\mu}v+\p^2_{\alpha\mu}u_{ap})
-\frac12Q^{\alpha\beta\mu}\p_{\alpha}\p_t^iv\p_{\beta}\p_t^iv(\p^2_{0\mu}v+\p^2_{0\mu}u_{ap}).
\end{equation}
Substituting \eqref{app-pw}, \eqref{BA} into \eqref{energy:time5} and \eqref{energy:time6} with $i\le2N$ implies
\begin{equation*}
\begin{split}
\sum_{i\le2N}\|\p\p_t^iv(t)\|^2_{L^2(\cK)}&\ls\sum_{i\le2N}\|\p\p_t^iv(0)\|^2_{L^2(\cK)}
+\sum_{i\le2N}\int_0^t\|\p_t\p_t^iv(s)\|_{L^2(\cK)}\|\p_t^iR_{ap}\|_{L^2(\cK)}ds\\
&+\ve_1\sum_{i\le2N}\sum_{j\le2N-1}\int_0^t\|\p\p_t^iv(s)\|_{L^2(\cK)}\|\p^{\le1}\p\p_t^jv(s)\|_{L^2(\cK)}\frac{ds}{\w{s}^{1/2}}.
\end{split}
\end{equation*}
Let $\ds E(t)=\sup_{s\in[0,t]}\sum_{i\le2N}\|\p\p_t^iv(s)\|_{L^2(\cK)}$. We then have
\begin{equation*}
\begin{split}
E(t)&\ls\ve^2+\ve_1\sum_{j\le2N-1}\int_0^t\|\p^{\le1}\p\p_t^jv(s)\|_{L^2(\cK)}\frac{ds}{\w{s}^{1/2}}
+\sum_{i\le2N}\int_0^t\|\p_t^iR_{ap}\|_{L^2(\cK)}ds\\
&\ls\ve^{3/2-\tilde\ve}+\ve_1\sum_{j\le2N-1}\int_0^t\|\p^{\le1}\p\p_t^jv(s)\|_{L^2(\cK)}\frac{ds}{\w{s}^{1/2}},
\end{split}
\end{equation*}
where  $v(0,x)=O(\ve^2)$ and \eqref{Rapp-L2} are used.
This completes the proof of \eqref{energy:time}.
\end{proof}

The following lemma is quite different from \cite[Lemma 4.2]{HouYinYuan26} and will play the key roles
in the local energy decay estimates \eqref{loc:energyB} and the energy estimates \eqref{energyB} below.
\begin{lemma}
Under the conditions of Theorem \ref{thm1} and assumption \eqref{BA}, it holds that for $t\in[0,\f{B^2}{\ve^2}-1]$,
\begin{equation}\label{loc:energyA}
\|\p^{\le1}v\|_{L^2(\cK_R)}\ls\ve^{3/2-2\tilde\ve}\w{t}^{-1/2-\tilde\ve}.
\end{equation}
\end{lemma}
\begin{proof}
Note that the equation \eqref{QWE} can be reformulated in the divergence form
\begin{equation}\label{eqn:div}
\Box u=-\frac12\sum_{\alpha,\beta,\mu=0}^2[\p_{\alpha}(Q^{\alpha\beta\mu}\p_{\beta}u\p_{\mu}u)
+\p_{\beta}(Q^{\alpha\beta\mu}\p_{\alpha}u\p_{\mu}u)
-\p_{\mu}(Q^{\alpha\beta\mu}\p_{\alpha}u\p_{\beta}u)].
\end{equation}

When $\ve t\le1$, one can find that $u_{ap}=\ve w_{B}$ and $\Box u_{ap}=\ve\Box w_{B}\equiv0$.
Applying \eqref{loc:ibvp} to $\Box v=\Box u$ in \eqref{eqn:div} with $\eta=\tilde\ve/2$ leads to
\begin{equation}\label{loc:energyA1}
\begin{split}
\w{t}\|\p^{\le1}v\|_{L^2(\cK_R)}&\ls\ve^2\w{t}^{\tilde\ve}+\ln(2+t)\sum_{|a|\le1}\sup_{s\in[0,t]}\w{s}\|\p^a[(\p u)^2]\|_{L^2(\cK_3)}\\
&\quad+\w{t}^{\tilde\ve}\sum_{|a|\le2}\sup_{(s,y)\in[0,t]\times\cK}\w{y}^{1/2}\cW_{1,1}(s,y)|Z^a[(\p u)^2]|\\
&\ls\ve^2\w{t}^{\tilde\ve}+\ve^{1.4}\ve_1\w{t}^{1/2+3\tilde\ve}+\ve^2\w{t}^{1/2+\tilde\ve},
\end{split}
\end{equation}
where $v(0,x)=O(\ve^2)$, \eqref{app-pw} and \eqref{BA} are used.
Then it follows from \eqref{loc:energyA1}, $\ve_1=C_0\ve$ and the smallness of $\ve_0$ that
\begin{equation}\label{loc:energyA2}
\|\p^{\le1}v\|_{L^2(\cK_R)}\ls\ve^2\w{t}^{\tilde\ve-1}+\ve^{2.4}C_0\w{t}^{3\tilde\ve-1/2}+\ve^2\w{t}^{\tilde\ve-1/2},
\end{equation}
which yields \eqref{loc:energyA} for $\ve t\le1$.

When $\ve t\ge1$, utilizing \eqref{loc:ibvp} to $\Box u$ instead of $\Box v$ with $\eta=\tilde\ve/2$ gives
\begin{equation}\label{loc:energyA3}
\begin{split}
\w{t}\|\p^{\le1}u\|_{L^2(\cK_R)}&\ls\ve\w{t}^{\tilde\ve}+\ln(2+t)\sum_{|a|\le1}\sup_{s\in[0,t]}\w{s}\|\p^a[(\p u)^2]\|_{L^2(\cK_3)}\\
&\quad+\w{t}^{\tilde\ve}\sum_{|a|\le2}\sup_{(s,y)\in[0,t]\times\cK}\w{y}^{1/2}\cW_{1,1}(s,y)|Z^a[(\p u)^2]|\\
&\ls\ve\w{t}^{\tilde\ve}+\ve^{1.4}\ve_1\w{t}^{1/2+3\tilde\ve}+\ve^2\w{t}^{1/2+\tilde\ve}.
\end{split}
\end{equation}
Note that when $\ve t\ge2$, $u_{ap}=\ve\chi(-3\ve\sigma)r^{-1/2}U(\tau,\sigma,\theta)$ vanishes for $|x|\le R$
($R>1$ is any fixed constant) and small $\ve$.
Collecting \eqref{loc:energyA2} and \eqref{loc:energyA3} with \eqref{app-pw} completes the proof of \eqref{loc:energyA}.
\end{proof}

\begin{lemma}
Under the conditions of Theorem \ref{thm1} and assumption \eqref{BA},  we have for $t\in[0,B^2/\ve^2-1]$,
\begin{equation}\label{energy:dtdx}
\sum_{|a|\le2N}\|\p\p^av\|_{L^2(\cK_R)}\ls\ve^{3/2-3\tilde\ve}.
\end{equation}
\end{lemma}
\begin{proof}
Set $\ds E_j(t)=\sum_{k=0}^{2N-j}\|\p\p_t^kv(t)\|_{H^j(\cK)}$ with $0\le j\le2N$.
Then one has that for $j\ge1$,
\begin{equation}\label{energy:dtdx1}
\begin{split}
E_j(t)&\ls\sum_{k=0}^{2N-j}[\|\p\p_t^kv(t)\|_{L^2(\cK)}
+\sum_{1\le|a|\le j}(\|\p_t\p_t^k\p_x^av(t)\|_{L^2(\cK)}+\|\p_x\p_t^k\p_x^av(t)\|_{L^2(\cK)})]\\
&\ls E_0(t)+E_{j-1}(t)+\sum_{k=0}^{2N-j}\sum_{2\le|a|\le j+1}\|\p_t^k\p_x^av(t)\|_{L^2(\cK)},
\end{split}
\end{equation}
where we have used the fact of
\begin{equation*}
\sum_{k=0}^{2N-j}\sum_{1\le|a|\le j}\|\p_t\p_t^k\p_x^av(t)\|_{L^2(\cK)}
\ls\sum_{k=0}^{2N-j}\sum_{|a'|\le j-1}\|\p_x\p_t^{k+1}\p_x^{a'}v(t)\|_{L^2(\cK)}
\ls E_{j-1}(t).
\end{equation*}
For the last term in \eqref{energy:dtdx1}, it can be deduced from the elliptic estimate \eqref{ellip} that
\begin{equation}\label{energy:dtdx2}
\begin{split}
\|\p_t^k\p_x^av\|_{L^2(\cK)}&\ls\|\Delta_x\p_t^kv\|_{H^{|a|-2}(\cK)}+\|\p_t^kv\|_{H^{|a|-1}(\cK_{R+1})}\\
&\ls\|\p_t^{k+2}v\|_{H^{|a|-2}(\cK)}+\|\p_t^k\Box v\|_{H^{|a|-2}(\cK)}\\
&\quad+\|\p_t^kv\|_{L^2(\cK_{R+1})}+\sum_{1\le|b|\le|a|-1}\|\p_t^k\p_x^bv\|_{L^2(\cK_{R+1})},
\end{split}
\end{equation}
where $\Delta_x=\p_t^2-\Box$ has been used.
Meanwhile, it follows from  \eqref{error-eqn}, \eqref{app-pw}, \eqref{Rapp-L2} and \eqref{BA} with $|a|+k\le2N+1$ that
\begin{equation}\label{energy:dtdx3}
\begin{split}
\|\p_t^k\Box v\|_{H^{|a|-2}(\cK)}
&\ls\sum_{l\le k}\ve_1(\|\p_t^l\p v\|_{H^{|a|-1}(\cK)}+\|\p_t^{l+2}v\|_{H^{|a|-2}(\cK)})+\ve^{3/2}\\
&\ls\ve_1[E_j(t)+E_{j-1}(t)]+\ve^{3/2}.
\end{split}
\end{equation}
Combining \eqref{energy:dtdx1}-\eqref{energy:dtdx3} shows that for $j\ge1$,
\begin{equation}\label{energy:dtdx4}
E_j(t)\ls E_0(t)+E_{j-1}(t)+\ve_1E_j(t)+\|v\|_{L^2(\cK_{R+1})}.
\end{equation}
Substituting \eqref{loc:energyA} and \eqref{energy:dtdx4} with $j=1$ and the smallness of $\ve_1$ into \eqref{energy:time} leads to
\begin{equation*}
E(t)=\sup_{s\in[0,t]}\sum_{i\le2N}\|\p\p_t^iv(s)\|_{L^2(\cK)}\ls\ve^{3/2-2\tilde\ve}+\ve_1\int_0^t\frac{E(s)ds}{\w{s}^{1/2}}.
\end{equation*}
Applying Lemma \ref{lem:Gronwall} to the above inequality with \eqref{BA1} yields
\begin{equation*}
E(t)\ls\ve^{3/2-2\tilde\ve}e^{\ve_1\w{t}^{1/2}}\ls\ve^{3/2-2\tilde\ve}e^{C_0}\ls\ve^{3/2-3\tilde\ve},
\end{equation*}
which completes the proof of \eqref{energy:dtdx}.
\end{proof}

\begin{lemma}
Under the conditions of Theorem \ref{thm1} and assumption \eqref{BA}, we have for $t\in[0,B^2/\ve^2-1]$,
\begin{equation}\label{energyA}
\sum_{|a|\le2N-1}\|\p Z^av\|_{L^2(\cK)}\ls\ve^{3/2-4\tilde\ve}\w{t}^{1/2}.
\end{equation}
\end{lemma}
\begin{proof}
Analogously to Lemma \ref{lem:energy:time}, one has that for any $|a|\le2N-1$,
\begin{equation}\label{energyA1}
\begin{split}
&\sum_{|a|\le2N-1}\|\p Z^av(t)\|^2_{L^2(\cK)}\ls\ve^4+\ve_1\sum_{|a|\le2N-1}\int_0^t\|\p Z^av(s)\|^2_{L^2(\cK)}\frac{ds}{\w{s}^{1/2}}\\
&\qquad+\sum_{|a|\le2N-1}\int_0^t\|\p Z^av(s)\|_{L^2(\cK)}\|Z^aR_{ap}\|_{L^2(\cK)}ds+|\cB^a_1|+|\cB^a_2|,\\
&\cB^a_1=\int_0^t\int_{\p\cK}(n_1(x),n_2(x))\cdot\nabla Z^av(s,x)\p_tZ^av(s,x)d\sigma ds,\\
&\cB^a_2=\int_0^t\int_{\p\cK}\sum_{j=1}^2\sum_{\beta,\mu=0}^2Q^{j\beta\mu}n_j(x)\p_tZ^av\p_{\beta}Z^av(\p_{\mu}v+\p_{\mu}u_{ap})d\sigma ds,
\end{split}
\end{equation}
where $d\sigma$ is the curve measure on $\p\cK$.
According to $\p\cK\subset\overline{\cK_1}$ and the trace theorem, we arrive at
\begin{equation}\label{energyA2}
\begin{split}
|\cB^a_1|&\ls\sum_{|b|\le|a|}\int_0^t\|(1-\chi_{[1,2]}(x))\p_t\p^bv\|_{L^2(\p\cK)}\|(1-\chi_{[1,2]}(x))\p_x\p^bv\|_{L^2(\p\cK)}ds\\
&\ls\sum_{|b|\le|a|}\int_0^t\|(1-\chi_{[1,2]}(x))\p \p^bv\|^2_{H^1(\cK)}ds\\
&\ls\sum_{|b|\le|a|+1}\int_0^t\|\p \p^bv\|^2_{L^2(\cK_2)}ds\ls\ve^{3-6\tilde\ve}(1+t),
\end{split}
\end{equation}
where \eqref{energy:dtdx} has been used. Analogously, $\cB^a_2$ can be also treated.
Then, \eqref{energyA1} and \eqref{energyA2} ensure that
\begin{equation*}
\begin{split}
\sum_{|a|\le2N-1}\|\p Z^av(t)\|^2_{L^2(\cK)}&\ls\ve^{3-6\tilde\ve}\w{t}
+\ve_1\sum_{|a|\le2N-1}\int_0^t\|\p Z^av(s)\|^2_{L^2(\cK)}\frac{ds}{\w{s}^{1/2}}\\
&\quad+\sum_{|a|\le2N-1}\int_0^t\|\p Z^av(s)\|_{L^2(\cK)}\|Z^aR_{ap}\|_{L^2(\cK)}ds.
\end{split}
\end{equation*}
Define $\ds\tilde E(t)=\sum_{|a|\le2N-1}\sup_{s\in[0,t]}\max\{\|\p Z^av(s)\|_{L^2(\cK)},\ve^{3/2-3\tilde\ve}\w{s}^{1/2}\}$. Then
\begin{equation}\label{energyA3}
\begin{split}
\tilde E(t)&\ls\ve^{3/2-3\tilde\ve}\w{t}^{1/2}+\ve_1\int_0^t\tilde E(s)\frac{ds}{\w{s}^{1/2}}+\sum_{|a|\le2N-1}\int_0^t\|Z^aR_{ap}\|_{L^2(\cK)}ds\\
&\ls\ve^{3/2-3\tilde\ve}\w{t}^{1/2}+\ve_1\int_0^t\tilde E(s)\frac{ds}{\w{s}^{1/2}},
\end{split}
\end{equation}
where we have used \eqref{Rapp-L2}.
Applying Lemma \ref{lem:Gronwall} to \eqref{energyA3} with \eqref{BA1} yields \eqref{energyA}.
\end{proof}

\subsection{Local energy decay estimates and improved energy estimates}

To improve the energy estimate in \eqref{energyA}, which grows rapidly in the time rate $\w{t}^{1/2}$, one needs to establish
a better estimate on the boundary term \eqref{energyA2}. This will be achieved by the following local energy decay estimates.
\begin{lemma}
Under the conditions of Theorem \ref{thm1} and assumption \eqref{BA}, it holds that for $t\in[0,\f{B^2}{\ve^2}-1]$,
\begin{equation}\label{loc:energyB}
\sum_{|a|\le2N-4}\|\p^av\|_{L^2(\cK_R)}\ls\ve^{3/2-13\tilde\ve}\w{t}^{-1/2-\tilde\ve}.
\end{equation}
\end{lemma}
\begin{proof}
Although the proof of \eqref{loc:energyB} is similar to that of \eqref{loc:energyA}, we still give the details for
completeness.
Applying $\p_t^j$ to \eqref{eqn:div} gives
\begin{equation}\label{loc:energyB1}
\Box\p_t^ju=-\frac12\sum_{\alpha,\beta,\mu=0}^2[\p_{\alpha}\p_t^j(Q^{\alpha\beta\mu}\p_{\beta}u\p_{\mu}u)
+\p_{\beta}\p_t^j(Q^{\alpha\beta\mu}\p_{\alpha}u\p_{\mu}u)
-\p_{\mu}\p_t^j(Q^{\alpha\beta\mu}\p_{\alpha}u\p_{\beta}u)].
\end{equation}
On the other hand, \eqref{Sobo:ineq} and \eqref{energyA} lead to
\begin{equation}\label{loc:energyB2}
\sum_{|a|\le2N-3}|\p Z^av|\ls\ve^{3/2-4\tilde\ve}\w{x}^{-1/2}\w{t}^{1/2}.
\end{equation}
When $\ve t\le1$, one has $\Box u_{ap}\equiv0$.
Utilizing \eqref{loc:ibvp} to $\Box\p_t^jv=\Box\p_t^ju$ in \eqref{loc:energyB1} with $j\le2N-5$ and $\eta=\tilde\ve/2$ leads to
\begin{equation}\label{loc:energyB3}
\begin{split}
\w{t}\|\p^{\le1}\p_t^jv\|_{L^2(\cK_{R_1})}&\ls\ve^2\w{t}^{\tilde\ve}
+\ln(2+t)\sum_{|a|\le1}\sup_{s\in[0,t]}\w{s}\|\p^a\p_t^j[(\p u)^2]\|_{L^2(\cK_3)}\\
&\quad+\w{t}^{\tilde\ve}\sum_{|a|\le2}\sup_{(s,y)\in[0,t]\times\cK}\w{y}^{1/2}\cW_{1,1}(s,y)|Z^a\p_t^j[(\p u)^2]|\\
&\ls\ve^2\w{t}^{\tilde\ve}+\ve^{3/2-4\tilde\ve}\ve_1\w{t}^{1+3\tilde\ve}+\ve^2\w{t}^{1/2+\tilde\ve}\\
&\ls\ve^{3/2}\w{t}^{1/2-\tilde\ve},
\end{split}
\end{equation}
where we have used \eqref{app-pw}, \eqref{BA}, \eqref{energy:dtdx} and \eqref{loc:energyB2}.

When $\ve t\ge1$, applying \eqref{loc:ibvp} to $\Box\p_t^ju$ in \eqref{loc:energyB1} with $\eta=\tilde\ve/2$ yields
\begin{equation*}
\w{t}\|\p^{\le1}\p_t^ju\|_{L^2(\cK_{R_1})}\ls\ve\w{t}^{\tilde\ve}+\ve^{3/2-4\tilde\ve}\ve_1\w{t}^{1+3\tilde\ve}+\ve^2\w{t}^{1/2+\tilde\ve}.
\end{equation*}
It is noticed that when $\ve t\ge2$, $u_{ap}=\ve\chi(-3\ve\sigma)r^{-1/2}U(\tau,\sigma,\theta)$ vanishes for $|x|\le R_1$ and small $\ve$.
Thus, by \eqref{app-pw}, we obtain that for $t\le\frac{B^2}{\ve^2}-1$,
\begin{equation}\label{loc:energyB4}
\begin{split}
\|\p^{\le1}\p_t^jv\|_{L^2(\cK_{R_1})}&\ls\ve\w{t}^{\tilde\ve-1}+\ve^{5/2-5\tilde\ve}C_0\ve_0^{\tilde\ve}\w{t}^{3\tilde\ve}+\ve^2\w{t}^{\tilde\ve-1/2}\\
&\ls\ve^{3/2-13\tilde\ve}\w{t}^{-1/2-\tilde\ve}.
\end{split}
\end{equation}
For $j=0,1,2,\cdots,2N-4$, denote $\ds E_j^{loc}(t):=\sum_{k\le j}\|\p_t^k\p_x^{2N-4-j}v\|_{L^2(\cK_{R+j})}$.
Then it follows from \eqref{loc:energyB4} that for $t\le\frac{B^2}{\ve^2}-1$,
\begin{equation}\label{loc:energyB5}
E_{2N-4}^{loc}(t)+E_{2N-5}^{loc}(t)\ls\ve^{3/2-13\tilde\ve}\w{t}^{-1/2-\tilde\ve}.
\end{equation}
When $j\le2N-6$, we have $\p_x^{2N-4-j}=\p_x^2\p_x^{2N-6-j}$.
Thus, one can apply the elliptic estimate \eqref{ellip} to $(1-\chi_{[R+j,R+j+1]})\p_t^{k}v$
to obtain
\begin{equation}\label{loc:energyB6}
\begin{split}
E_j^{loc}(t)&\ls\sum_{k\le j}\|\p_x^{2N-4-j}(1-\chi_{[R+j,R+j+1]})\p_t^kv\|_{L^2(\cK)}\\
&\ls\sum_{k\le j}[\|\Delta(1-\chi_{[R+j,R+j+1]})\p_t^kv\|_{H^{2N-6-j}(\cK)}\\
&\quad+\sum_{k\le j}\|(1-\chi_{[R+j,R+j+1]})\p_t^kv\|_{H^{2N-5-j}(\cK)}]\\
&\ls\sum_{k\le j}\|(1-\chi_{[R+j,R+j+1]})\Delta\p_t^kv\|_{H^{2N-6-j}(\cK)}+\sum_{j+1\le l\le2N-4}E_l^{loc}(t)\\
&\ls\sum_{k\le j}[\|\Box\p_t^kv\|_{H^{2N-6-j}(\cK_{R+j+1})}+\|\p_t^{k+2}v\|_{H^{2N-6-j}(\cK_{R+j+1})}]\\
&\quad+\sum_{j+1\le l\le2N-4}E_l^{loc}(t),
\end{split}
\end{equation}
where the fact of $\Delta=\p_t^2-\Box$ has been used.
Note that when $\ve t\ge2$, $u_{ap}=\ve\chi(-3\ve\sigma)r^{-1/2}U(\tau,\sigma,\theta)$ vanishes for $|x|\le R+2N$ and small $\ve$.
From \eqref{error-eqn}, \eqref{app-pw}, \eqref{Rapp-L2}, \eqref{BA} and \eqref{energy:dtdx}, one can get that for $j\le2N-6$,
\begin{equation}\label{loc:energyB7}
\sum_{k\le j}\|\Box\p_t^kv\|_{H^{2N-6-j}(\cK_{R+j+1})}\ls\ve^{3/2-13\tilde\ve}\w{t}^{-1/2-\tilde\ve}.
\end{equation}
Thus, \eqref{loc:energyB5}, \eqref{loc:energyB6} and \eqref{loc:energyB7} imply that for $j\le2N-6$,
\begin{equation*}
E_j^{loc}(t)\ls\ve^{3/2-13\tilde\ve}\w{t}^{-1/2-\tilde\ve}+\sum_{j+1\le l\le2N-4}E_l^{loc}(t).
\end{equation*}
This, together with \eqref{loc:energyB5}, yields \eqref{loc:energyB}.
\end{proof}

\begin{lemma}\label{YH-10}
Under the conditions of Theorem \ref{thm1} and assumption \eqref{BA}, we have for $t\in[0,B^2/\ve^2-1]$,
\begin{equation}\label{energyB}
\sum_{|a|\le2N-6}\|\p Z^av\|_{L^2(\cK)}\ls\ve^{3/2-14\tilde\ve}.
\end{equation}
\end{lemma}
\begin{proof}
Analogously to \eqref{energyA1} and \eqref{energyA2}, one has that for any $|a|\le2N-6$,
\begin{equation}\label{energyB1}
\begin{split}
&\sum_{|a|\le2N-1}\|\p Z^av(t)\|^2_{L^2(\cK)}\ls\ve^4+\ve_1\sum_{|a|\le2N-1}\int_0^t\|\p Z^av(s)\|^2_{L^2(\cK)}\frac{ds}{\w{s}^{1/2}}\\
&\qquad+\sum_{|a|\le2N-1}\int_0^t\|\p Z^av(s)\|_{L^2(\cK)}\|Z^aR_{ap}\|_{L^2(\cK)}ds+|\cB^a_1|+|\cB^a_2|
\end{split}
\end{equation}
and
\begin{equation}\label{energyB2}
|\cB^a_1|\ls\sum_{|b|\le|a|+1}\int_0^t\|\p \p^bv\|^2_{L^2(\cK_2)}ds
\ls\int_0^t\ve^{3-26\tilde\ve}(1+s)^{-1-2\tilde\ve}ds
\ls\ve^{3-26\tilde\ve},
\end{equation}
where \eqref{loc:energyB} has been used.
Similarly, $\cB^a_2$ can be also treated.
Then collecting \eqref{energyB1} and \eqref{energyB2} yields
\begin{equation*}
\begin{split}
&\sum_{|a|\le2N-6}\|\p Z^av(t)\|^2_{L^2(\cK)}\ls\ve^{3-26\tilde\ve}+\ve_1\sum_{|a|\le2N-1}\int_0^t\|\p Z^av(s)\|^2_{L^2(\cK)}\frac{ds}{\w{s}^{1/2}}\\
&\qquad+\sum_{|a|\le2N-1}\int_0^t\|\p Z^av(s)\|_{L^2(\cK)}\|Z^aR_{ap}\|_{L^2(\cK)}ds,
\end{split}
\end{equation*}
which implies \eqref{energyB}.
\end{proof}

\subsection{Improved pointwise estimates and proof of Theorem \ref{thm1}}

\begin{lemma}
Under the conditions of Theorem \ref{thm1} and assumption \eqref{BA}, it holds that for $t\in[0,\f{B^2}{\ve^2}-1]$,
\begin{equation}\label{loc:energyC}
\sum_{|a|\le2N-9}\|\p^av\|_{L^2(\cK_R)}\ls\ve\w{t}^{\tilde\ve-1}.
\end{equation}
\end{lemma}
\begin{proof}
The proof of \eqref{loc:energyC} is actually  much easier than that of \eqref{loc:energyB}.
Note that \eqref{Sobo:ineq} and \eqref{energyB} lead to
\begin{equation}\label{loc:energyC1}
\sum_{|a|\le2N-8}|\p Z^av|\ls\ve^{3/2-14\tilde\ve}\w{x}^{-1/2}.
\end{equation}
Applying \eqref{loc:ibvp} to $\Box\p_t^ju$ in \eqref{loc:energyB1} with $j\le2N-10$, $\eta=\tilde\ve/2$ and \eqref{loc:energyC1} instead of \eqref{loc:energyB2} implies
\begin{equation*}
\begin{split}
\w{t}\|\p^{\le1}\p_t^ju\|_{L^2(\cK_{R_1})}&\ls\ve\w{t}^{\tilde\ve}+\ve^{3/2-14\tilde\ve}\ve_1\w{t}^{1/2+3\tilde\ve}+\ve^2\w{t}^{1/2+\tilde\ve}.
\end{split}
\end{equation*}
This, together with an analogous argument in the proof of \eqref{loc:energyB}, yields the proof of \eqref{loc:energyC}.
\end{proof}

\begin{lemma}
Under the conditions of Theorem \ref{thm1} and assumption \eqref{BA}, it holds that for $t\in[0,\f{B^2}{\ve^2}-1]$,
\begin{equation}\label{BA:impr}
\sum_{|a|\le2N-19}|\p Z^av|\ls\ve\w{x}^{-1/2}\w{t-|x|}^{-1}\w{t}^{2\tilde\ve}.
\end{equation}
\end{lemma}
\begin{proof}
Applying $Z^a$ to \eqref{eqn:div} yields
\begin{equation}\label{BA:impr1}
\Box Z^au=\sum_{b+c\le a}\sum_{\alpha,\beta,\mu=0}^2\p_{\alpha}(Q_{abc}^{\alpha\beta\mu}\p_{\beta}Z^bu\p_{\mu}Z^cu),
\end{equation}
where $Q_{abc}^{\alpha\beta\mu}$ are constants.
Set $\tilde Z=\chi_{[1/2,1]}(x)Z$. Then $\ds\cB\tilde Z^au|_{[0,\infty)\times\p\cK}=0$ holds.

By using \eqref{dpw:ibvp} to $\Box\tilde Z^au=\Box(\tilde Z^a-Z^a)u+\Box Z^au$ with $\eta=\tilde\ve$ and $|a|\le2N-19$, one can achieve
\begin{equation}\label{BA:impr2}
\begin{split}
&\quad\w{x}^{1/2}\w{t-|x|}|\p\tilde Z^au|
\ls\ve\w{t}^{\tilde\ve}\ln(2+t)\\
&\qquad+\ln(2+t)\sum_{|b|\le8}\sup_{s\in[0,t]}\w{s}\|\p^b(G^r,\p^{\le1}G^\alpha)(s)\|_{L^2(\cK_3)}\\
&\qquad+\w{t}^{\tilde\ve}\ln^2(2+t)\sum_{|b|\le10}\sup_{(s,y)\in[0,t]\times\overline{\R^2\setminus\cK_2}}\w{y}^{1/2}\cW_{1,1}(s,y)|Z^bG^\alpha(s,y)|,
\end{split}
\end{equation}
where we have used \eqref{initial:data}, $\supp_x\Box(\tilde Z^a-Z^a)u\subset\{x: |x|\le1\}$  and the expressions of $G^{\alpha},G^r$ by
\begin{equation}\label{BA:impr3}
G^{\alpha}=\sum_{\tilde a+\tilde b\le a}|Z^{\tilde a}\p uZ^{\tilde b}\p u|,\quad G^r=\Box(\tilde Z^a-Z^a)u.
\end{equation}
It follows from $u=v+u_{ap}$, \eqref{app-pw} and \eqref{loc:energyC} that
\begin{equation}\label{BA:impr4}
\sum_{|b|\le8}\sup_{s\in[0,t]}\w{s}\|\p^b(G^r,\p^{\le1}G^\alpha)(s)\|_{L^2(\cK_3)}\ls\ve\w{t}^{\tilde\ve}.
\end{equation}
On the other hand, \eqref{app-pw}, \eqref{BA} and \eqref{loc:energyC1} imply
\begin{equation}\label{BA:impr5}
\sum_{|b|\le10}\sup_{(s,y)\in[0,t]\times\overline{\R^2\setminus\cK_2}}\w{y}^{1/2}\cW_{1,1}(s,y)|Z^bG^\alpha(s,y)|
\ls\ve^{3/2-14\tilde\ve}\ve_1\w{t}^{1/2+2\tilde\ve}+\ve^2\w{t}^{1/2}\ls\ve.
\end{equation}
Collecting \eqref{BA:impr2}-\eqref{BA:impr5} leads to
\begin{equation*}
\w{x}^{1/2}\w{t-|x|}|\p\tilde Z^au|\ls\ve\w{t}^{2\tilde\ve}.
\end{equation*}
This, together with \eqref{app-pw}, \eqref{loc:energyC} and the standard Sobolev embedding, yields \eqref{BA:impr}.
\end{proof}

Finally, we turn to prove Theorem \ref{thm1}.

\begin{proof}[Proof of Theorem \ref{thm1}]
At first, \eqref{loc:energyC1} and \eqref{BA:impr} ensure that there is $C_1\ge1$ (which depends on $B$) such that
\begin{equation*}
\begin{split}
\sum_{|a|\le2N-8}|\p Z^av|&\le C_1\ve^{3/2-14\tilde\ve}\w{x}^{-1/2},\\
\sum_{|a|\le2N-19}|\p Z^av|&\le C_1\ve\w{x}^{-1/2}\w{t-|x|}^{-1}\w{t}^{2\tilde\ve}.
\end{split}
\end{equation*}
Choosing $\ve_1=2C_1\ve$ and $\ve_0=e^{-2C_110^3}$ such that \eqref{BA1} is satisfied.
For $t\le B^2/\ve^2-1$ and $N\ge19$, \eqref{BA} can be improved to
\begin{equation*}
\begin{split}
\sum_{|a|\le N}|\p Z^av|&\le\frac{1}{2}\ve^{1.4}\w{x}^{-1/2},\\
\sum_{|a|\le N}|\p Z^av|&\le\frac{1}{2}\ve_1\w{x}^{-1/2}\w{t-|x|}^{-1}\w{t}^{2\tilde\ve}.
\end{split}
\end{equation*}
This, together with the local existence of classical solution to the initial boundary value problem
of the nonlinear wave equation, implies that problem \eqref{QWE} admits a solution $u\in\bigcap\limits_{j=0}^{2N+1}C^{j}([0,\frac{B^2}{\ve^2}-1], H^{2N+1-j}(\cK))$.
Therefore, for any positive constant $B<\tau_0$, we have
\begin{equation*}
\liminf_{\ve\rightarrow 0+}\ve\sqrt{T_\ve}\ge B.
\end{equation*}
Letting $B\rightarrow\tau_0-$ yields the proof of Theorem \ref{thm1}.
\end{proof}

\section{The upper bound on the lifespan $T_\ve$ and proof of Theorem \ref{thm2}}\label{sect6}

This section aims to establish the upper bound on the lifespan $T_\ve$. It is pointed out that 
some ideas are inspired by \cite{Hormander97book,John}.

\subsection{Setup for the proof of Theorem \ref{thm2}}

Let $u=u(t,r)$ be the radial symmetric solution of \eqref{rad:wave}. Then one can see that
\begin{equation}\label{rad-eqn}
(\p_t^2-c^2\p_r^2)(\sqrt{r}u)=\frac{c^2}{4r^{3/2}}u,
\end{equation}
where the notation $c=c(\p_tu)$ is used in this section.
Define 
\begin{equation}\label{W-def}
\begin{split}
L_1:=\p_t+c\p_r,\qquad\qquad\quad &L_2:=\p_t-c\p_r,\\
W_1:=-\frac{1}{2c}L_2\p_r(\sqrt{r}u),\qquad &W_2:=\frac{1}{2c}L_1\p_r(\sqrt{r}u).
\end{split}
\end{equation}
It follows from \eqref{rad-eqn} and \eqref{W-def} that $W_1$ and $W_2$ satisfy
\begin{equation}\label{W-eqn}
\begin{split}
L_1W_1&=\frac{cc'}{\sqrt{r}}(W_1^2-W_1W_2)+\frac{cc'}{8r^2}u(W_1-W_2)
+\frac{c'}{4r}\p_tu(3W_1+W_2)\\
&\quad+\frac{c'}{8r^{5/2}}u\p_tu+\frac{3c}{16 r^{5/2}}u-\frac{c}{8r^{3/2}}\p_ru,\\
L_2W_2&=\frac{cc'}{\sqrt{r}}(W_2^2-W_1W_2)+\frac{cc'}{8r^2}u(W_2-W_1)-\frac{c'}{4r}\p_tu(W_1+3W_2)\\
&\quad-\frac{c'}{8 r^{5/2}}u\p_tu-\frac{3c}{16r^{5/2}}u+\frac{c}{8r^{3/2}}\p_ru,
\end{split}
\end{equation}
respectively.

\begin{lemma}
Let $W_1,W_2$ be defined by \eqref{W-def}, one then has
\begin{equation}\label{Diff-Form}
\begin{split}
d(|W_1|(dr-cdt))&=\sgn W_1\Big[\frac{cc'}{8r^2}u(W_1-W_2)+\frac{c'}{4r}\p_tu(W_1+W_2)\\
&\quad+\frac{c'}{8r^{5/2}}u\p_t u+\frac{3c}{16r^{5/2}}u-\frac{c}{8r^{3/2}}\p_ru\Big]dt\wedge dr,
\end{split}
\end{equation}
where $\sgn$ stands for the sign function.
\end{lemma}
\begin{proof}
\eqref{Diff-Form} can be obtained directly by \eqref{W-eqn} with $\p_rc=c'(\frac{c(W_2-W_1)}{\sqrt{r}}-\frac{\p_t u}{2r})$ and
\begin{equation*}
d(|W_1|(dr-cdt))=\sgn W_1(L_1W_1+W_1\p_rc)dt\wedge dr.
\end{equation*}
\end{proof}
Assume that $\p_\sigma^2F_0(\sigma_0)=\max\limits_{\sigma}\p_\sigma^2F_0(\sigma)$.
For any $\lambda\in[\sigma_0-1,M_0]$ and $\mu\in\R$, let $\Gamma^+_\lambda$ and $\Gamma^-_\mu$ be the forward 
and backward characteristic curves (see Figure \ref{fig1}) passing through $(t_0,t_0+\lambda)$ and $(t_0,-t_0+\mu)$ 
in the $(t,r)$-plane, respectively,
\begin{equation}\label{chara-curve}
\begin{split}
&\frac{dr(t)}{dt}=c(\p_tu(t,r(t))),\qquad\qquad r(t_0)=t_0+\lambda,\\
&\frac{d\tilde r(t)}{dt}=-c(\p_tu(t,\tilde r(t))),\qquad\quad \tilde r(t_0)=-t_0+\mu.
\end{split}
\end{equation}
Choosing a large positive constant $t_0$ such that $t_0+\sigma_0-1\ge2$ which ensures that $r(t_0)=t_0+\lambda$ 
does not reach the obstacle $\cO$ and \eqref{chara-curve} is well defined.
For any $B<\tau_0$ and small $\ve$, denote $D$ the region surrounded by
\begin{equation*}
\Gamma^+_{\sigma_0-1},\quad\Gamma^+_{M_0},\quad\{t=t_0\},\quad\{t=T_\ve=\frac{B^2}{\ve^2}-1\}.
\end{equation*}

\begin{figure}[!h]
  \centering
  \includegraphics[width=0.8\textwidth]{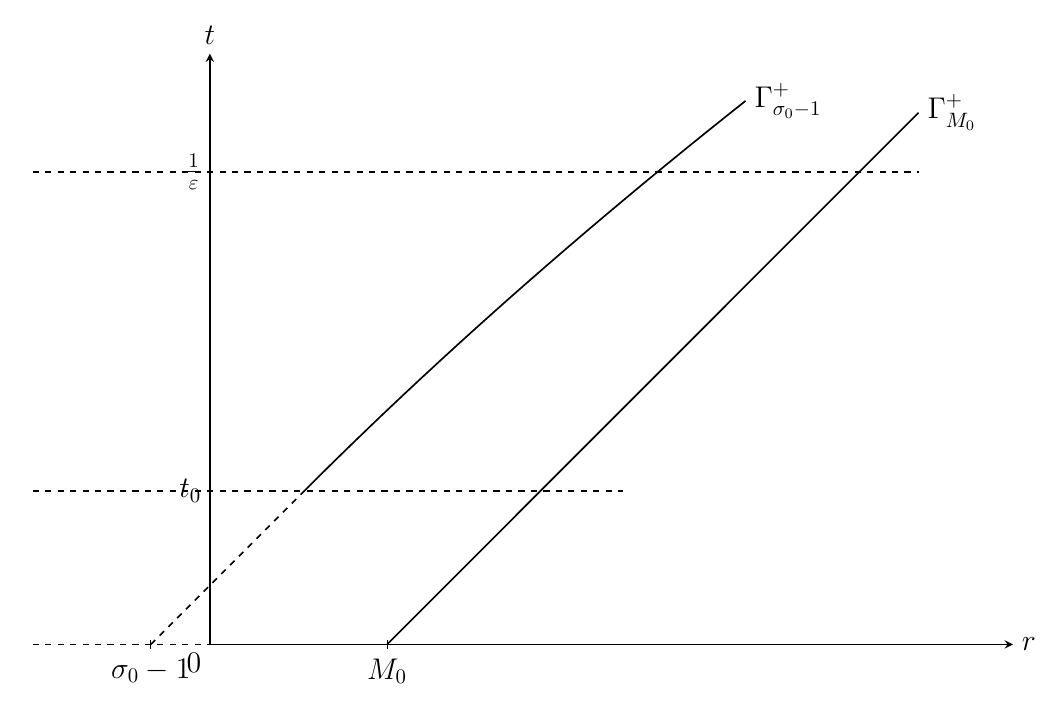}
  \caption{Characteristics}\label{fig1}
\end{figure}

\begin{lemma}\label{Lem6.2}
For any $\mu\in\R$ and $\lambda,\lambda'\in[\sigma_0-1,M_0]$, if $(t,r),(t',r')\in\Gamma^-_\mu\cap D$ and $(t,r)\in\Gamma^+_\lambda,(t',r')\in\Gamma^+_{\lambda'}$, then one has for $t\le T_\ve$,
\begin{equation}\label{time-lem}
|t-t'|\le C_B+|\lambda-\lambda'|.
\end{equation}
In particular, if $t,t'\le T_\ve':=\frac{\tau_0^2}{4\ve^2}-1$, the positive constant $C_B$ above and in the following part 
can be replaced by $C$ which is independent of $B$.
\end{lemma}
\begin{proof}
It follows from \eqref{app-pw}, \eqref{loc:energyC1} and \eqref{BA:impr} that for $t\le T_\ve$,
\begin{equation}\label{time-lem-pf1}
|\p\p^{\le1}u(t,r)|\le C_B\ve(1+t)^{-1/2}.
\end{equation}
This, together with \eqref{chara-curve}, yields
\begin{equation}\label{time-lem-pf2}
\begin{split}
|r(t)-t-\lambda|&=\Big|t_0+\lambda+\int_{t_0}^tc(\p_tu(s,r(s)))ds-t-\lambda\Big|\\
&\le\int_{t_0}^t|c(\p_tu(s,r(s)))-1|ds\\
&\le C_B\ve\int_{t_0}^t(1+s)^{-1/2}ds\le C_B,
\end{split}
\end{equation}
which means the horizontal width of $D$ is finite.
Analogously, one has
\begin{equation*}
\begin{split}
|\tilde r(t)+t-\mu|&=\Big|-t_0+\mu+\int_{t_0}^t-c(\p_tu(s,\tilde r(s)))ds+t-\mu\Big|\\
&\le\int_{t_0}^t|1-c(\p_tu(s,\tilde r(s)))|ds\\
&\le C_B\ve\int_{t_0}^t(1+s)^{-1/2}ds\le C_B.
\end{split}
\end{equation*}
On the other hand, it is easy to see that
\begin{equation*}
\begin{split}
|t-t'|&\le\frac{1}{2}(|t+r-\mu|+|t'+r'-\mu|+|r-t-\lambda|+|r'-t'-\lambda'|+|\lambda-\lambda'|)\\
&\le C_B+|\lambda-\lambda'|,
\end{split}
\end{equation*}
which finishes the proof of Lemma \ref{Lem6.2}.
\end{proof}

\subsection{Main estimates}

For $t\le T_\ve=\frac{B^2}{\ve^2}-1$, define
\begin{equation}\label{XYZ-def}
\begin{split}
X(t)&=\sup_{\frac{1}{\ve}\le s\le t}\int_{(s,r)\in D}|W_1(s,r)|dr, \\
Y(t)&=\sup_{\frac{1}{\ve}\le s\le t}\sup_{(s,r)\in D}s^{\frac{1}{2}}|\p_{s,r}u(s,r)|, \\
Z(t)&=\sup_{\frac{1}{\ve}\le s\le t}\sup_{(s,r)\in D}s|W_2(s,r)|.
\end{split}
\end{equation}

\begin{lemma}
There exits a constant $E\ge1$ independent of $B$ such that
\begin{equation}\label{XYZ-initial}
X(\frac{1}{\ve})\le\frac{E\ve}{2},\quad Y(\frac{1}{\ve})\le E\ve,\quad Z(\frac{1}{\ve})\le E^2\ve^2.
\end{equation}
\end{lemma}
\begin{proof}
By \eqref{time-lem-pf2}, one can see that
\begin{equation*}
X(\frac{1}{\ve})\le C\ve\int_{(\frac{1}{\ve},r)\in D}dr\le C\ve,
\end{equation*}
where and in the remaining part of this lemma, the generic constant $C>0$ is independent of $B$.
On the other hand, it follows from \eqref{time-lem-pf1} that $Y(\frac{1}{\ve})\le C\ve$.

Next, we turn our attention to $Z(\frac{1}{\ve})$.
For $t\in[\frac{1}{3\ve},\frac{3}{\ve}]$ and $(t,r)\in D$, \eqref{W-def}, \eqref{time-lem-pf1} and \eqref{time-lem-pf2} imply
\begin{equation*}
|u(t,r)|\le C\ve(1+t)^{-1/2},\quad |W_1(t,r)|+|W_2(t,r)|\le C\ve.
\end{equation*}
Then, by \eqref{W-eqn}, we can achieve
\begin{equation*}
|L_2W_2(t,r)|\le C(\ve^2t^{-1/2}+\ve t^{-2}),\qquad (t,r)\in D, \quad t\in[\frac{1}{3\ve},\frac{3}{\ve}].
\end{equation*}
For any $t\in[\frac{1}{2\ve},\frac{2}{\ve}]$, $(t,r)\in D$ and small $\ve$, there is $(t',r')\in D\cap\Gamma^+_{M_0}$ with $t'\in[\frac{1}{3\ve},\frac{3}{\ve}]$ such that $(t,r),(t',r')\in\Gamma^-_\mu$.
Then integrating $L_2W_2$ along the backward characteristic curve $\Gamma^-_\mu$ leads to
\begin{equation*}
|W_2(t,r)|\le\int_{(t',\tilde r(t'))}^{(t,\tilde r(t))}C(\ve^2s^{-1/2}+\ve s^{-2})\le C\ve^2t^{-1/2},
\end{equation*}
which yields $|L_2W_2(t,r)|\le C(\ve^3t^{-1}+\ve^2t^{-3/2}+\ve t^{-2})$ provided that $t\in[\frac{1}{2\ve},\frac{2}{\ve}]$ and $(t,r)\in D$.
For any $(t=\frac{1}{\ve},r)\in D$ and small $\ve$, repeating the above procedure completes the proof of \eqref{XYZ-initial}.
\end{proof}

\begin{lemma}\label{Lem6.4}
For $t\in[\frac{1}{\ve},T_\ve]$ with $\ve$ sufficiently small, under the assumptions
\begin{equation}\label{XYZ-assum}
X(t)\le E\ve,\quad Y(t)\le3E\ve,\quad Z(t)\le7\tilde CE^2\ve^2,
\end{equation}
then one has
\begin{equation}\label{XYZ-impr}
X(t)\le\frac{5E\ve}{8},\quad Y(t)\le2E\ve,\quad Z(t)\le6\tilde CE^2\ve^2,
\end{equation}
where $\tilde C=M_0+|\sigma_0-1|+\tau_0>1$.
\end{lemma}
\begin{proof}
Analogously to the proof of \eqref{time-lem} with \eqref{XYZ-assum} instead of \eqref{time-lem-pf1}, one has
\begin{equation}\label{XYZ-impr-pf1}
\begin{split}
|r(t)-t-\lambda|&\le4E\tau_0,\\
|\tilde r(t)+t-\mu|&\le4E\tau_0,\\
|t-t'|&\le8E\tau_0+|\lambda-\lambda'|\le 8\tilde CE,
\end{split}
\end{equation}
provided that $\mu\in\R$, $\lambda,\lambda'\in[\sigma_0-1,M_0]$, $(t,r),(t',r')\in\Gamma^-_\mu\cap D$ and $(t,r)\in\Gamma^+_\lambda,(t',r')\in\Gamma^+_{\lambda'}$.
The proof of \eqref{XYZ-impr} will be separated into the following three parts.

\subsubsection*{A. Estimate of $X(t)$}
By \eqref{Diff-Form}, one has
\begin{equation}\label{XYZ-impr-pf2}
\begin{split}
\int_{(t,r)\in D}|W_1(t,r)|dr
&\le\int_{(\frac{1}{\ve},r)\in D}|W_1(\frac{1}{\ve},r)|dr+\iint_{\substack{\frac{1}{\ve}\le s\le t \\ (s,r)\in D}}
\Big|\frac{cc'}{8r^2}u(W_1-W_2)\\
&\quad+\frac{c'}{4r}\p_t u(W_1+W_2)+\frac{c'}{8r^{5/2}}u\p_tu+\frac{3c}{16r^{5/2}}u-\frac{c}{8r^{3/2}}\p_ru\Big|dsdr.
\end{split}
\end{equation}
On the other hand, \eqref{XYZ-assum} and \eqref{XYZ-impr-pf1} yield $|u(t,r)|\le2E\ve t^{-1/2}(4E\tau_0+M_0+|\sigma_0-1|)$.
This, together with \eqref{XYZ-initial}, \eqref{XYZ-assum}, \eqref{XYZ-impr-pf2} and the smallness of $\ve$, implies
\begin{equation}\label{XYZ-impr-pf3}
\begin{split}
\int_{(t,r)\in D}|W_1(t,r)|dr
&\le\frac{E\ve}{2}+C\int_{\frac{1}{\ve}}^t\ve s^{-3/2}X(s)ds\\
&\quad+C\iint_{\substack{\frac{1}{\ve}\le s\le t \\ (s,r)\in D}}\ve(s^{-2}+s^{-5/2}Z(s))dsdr\\
&\le\frac{E\ve}{2}+C\ve^2\le\frac{5E\ve}{8},
\end{split}
\end{equation}
where the constant $C>0$ depends on $E,M_0,\sigma_0,\tau_0$ and is independent of $B$.
Thus, the estimate of $X(t)$ in \eqref{XYZ-impr} is finished.

\subsubsection*{B. Estimate of $Y(t)$}
For any $(t,r)\in D$ with $t\ge\frac{1}{\ve}$, there is $(t',r')\in\Gamma^+_{M_0}$ such that $(t,r),(t',r')\in\Gamma^-_{\mu}$.

When $t'\le\frac{1}{\ve}$, one can see that $t\le|t-t'|+t'\le\frac{2}{\ve}$.
Analogously to \eqref{time-lem-pf1}, one has $Y(t)\le C\ve$ with the constant $C>0$ independent of $B, E$.

When $t'\ge\frac{1}{\ve}$, set $\Gamma_t=\{(t,\breve{r}),\breve{r}\ge1\}$ and denote $D_1$ the domain (see Figure \ref{fig2}) surrounded by $\Gamma_t,\Gamma^+_{M_0},\Gamma^-_{\mu}$.

\begin{figure}[!h]
  \centering
  \includegraphics[width=0.7\textwidth]{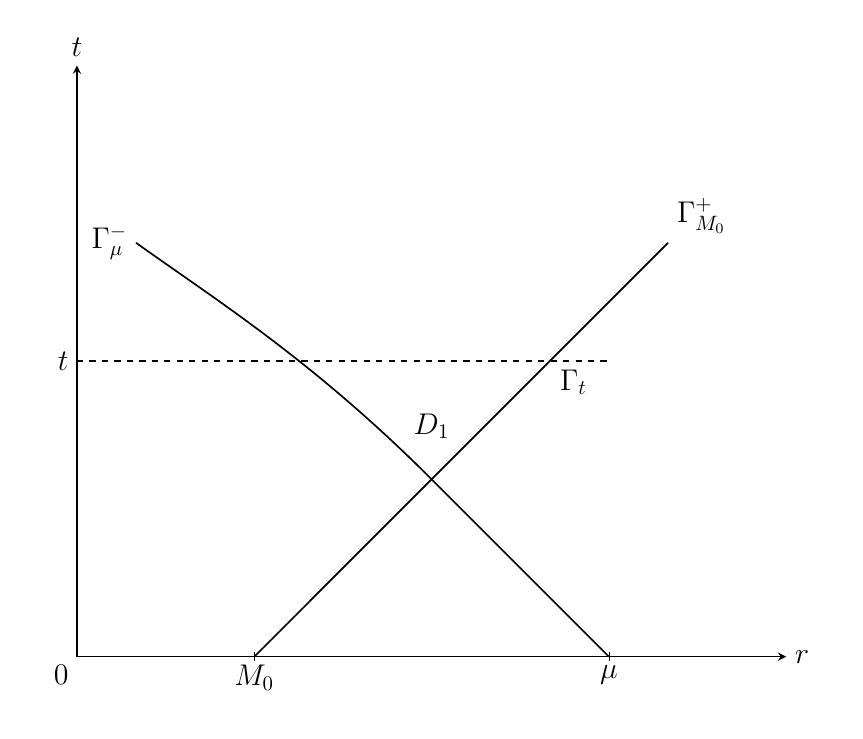}
  \caption{Domain $D_1$}\label{fig2}
\end{figure}

It follows from \eqref{rad-eqn} and \eqref{W-def} that
\begin{equation*}
\begin{split}
L_2\p_r(\sqrt{r}u)&=-2cW_1,\\
L_2\p_t(\sqrt{r}u)&=\p_t^2(\sqrt{r}u)-c\p_{tr}^2(\sqrt{r}u)
=\frac{c^2}{4r^{3/2}}u+c^2\p_r^2(\sqrt{r}u)-c\p_{tr}^2(\sqrt{r}u)\\
&=\frac{c^2}{4r^{3/2}}u-cL_2\p_r(\sqrt{r}u)=\frac{c^2}{4r^{3/2}}u+2c^2W_1.
\end{split}
\end{equation*}
Integrating these equations along the backward characteristic curve $\Gamma^-_{\mu}$ with $(t',r')\in\Gamma^+_{M_0}$ yields
\begin{equation}\label{XYZ-impr-pf4}
\p_{t,r}(\sqrt{r}u)=\int_{t'}^t L_2\p_{t,r}(\sqrt{\tilde{r}}(s)u(s,\tilde{r}(s)))ds.
\end{equation}
Thus, one needs to treat the estimate of $\int_{\Gamma^-_{\mu}}|W_1|(dr-cdt)$.
To this end, integrating \eqref{Diff-Form} over $D_1$ gives
\begin{equation*}
\begin{split}
(\int_{\Gamma_t}+\int_{\Gamma^-_{\mu}})|W_1|(dr-cdt)&=\iint_{(s,r)\in D_1}\sgn W_1\Big[\frac{cc'}{8r^2}u(W_1-W_2)+\frac{c'}{4r}\p_tu(W_1+W_2)\\
&\quad+\frac{c'}{8r^{5/2}}u\p_tu+\frac{3c}{16r^{5/2}}u-\frac{c}{8 r^{3/2}}\p_ru\Big]dsdr.
\end{split}
\end{equation*}
This, together with \eqref{XYZ-assum}, \eqref{XYZ-impr-pf1} and \eqref{XYZ-impr-pf3}, implies
\begin{equation}\label{XYZ-impr-pf5}
\begin{split}
\int_{\Gamma^-_{\mu}}|W_1|(dr-cdt)
&\le\iint_{(s,r)\in D_1}\Big|\frac{cc'}{8r^2}u(W_1-W_2)+\frac{c'}{4r}\p_tu(W_1+W_2)+\frac{c'}{8r^{5/2}}u\p_tu\\
&\quad+\frac{3c}{16r^{5/2}}u-\frac{c}{8 r^{3/2}}\p_ru\Big|dsdr+\int_{(t,r)\in D}|W_1(t,r)|dr\\
&\le C\int_{t'}^t(\ve^2s^{-3/2}+\ve s^{-2})ds+\frac{5E\ve}{8}\le\frac{3E\ve}{4}.
\end{split}
\end{equation}
Substituting \eqref{XYZ-impr-pf5} into \eqref{XYZ-impr-pf4} yields
\begin{equation*}
|\p_{t,r}(\sqrt{r}u)|\le\frac{7}{3}\int_{\Gamma^-_{\mu}}|W_1|(dr-cdt)+C\int_{t'}^t\ve s^{-2}ds\le\frac{15E\ve}{8},
\end{equation*}
which gives the estimate of $Y(t)$ in \eqref{XYZ-impr}.

\subsubsection*{C. Estimate of $Z(t)$}
For any $(t,r)\in D$ with $t\ge\frac{1}{\ve}$, $(t',r')\in\Gamma^+_{M_0}$ is the same point in the estimate of $Y(t)$.

When $t'\le\frac{1}{\ve}$, one has $t\le\frac{2}{\ve}$ and $Z(t)\le C\ve^2$ which is analogous to the proof of \eqref{XYZ-initial}.

\vskip 0.1cm

When $t'\ge\frac{1}{\ve}$, the second equation of \eqref{W-eqn} can be rewritten as
\begin{equation*}
\begin{split}
L_2W_2&=aW_2+b,\\
a:&=\frac{cc'}{\sqrt{r}}(W_2-W_1)+\frac{cc'}{8r^2}u-\frac{3c'}{4r}\p_tu,\\
b:&=-\frac{cc'}{8r^2}u W_1-\frac{c'}{4r}\p_tuW_1-\frac{c'}{8r^{5/2}}u\p_tu
-\frac{3c}{16r^{5/2}}u+\frac{c}{8r^{3/2}}\p_ru.
\end{split}
\end{equation*}
Integrating $L_2W_2$ along the backward characteristic curve $\Gamma^-_{\mu}$ leads to
\begin{equation*}
W_2(t,r)=\int_{t'}^tL_2W_2(s,\tilde{r}(s))ds.
\end{equation*}
In addition, it can be concluded from \eqref{XYZ-assum}, \eqref{XYZ-impr-pf1} and \eqref{XYZ-impr-pf5} that
\begin{equation*}
\begin{split}
\Big|\int_{t'}^tb(s,\tilde{r}(s))ds\Big|&\le \frac{C\ve}{t^{3/2}}\int_{\Gamma^-_{\mu}}|W_1|(dr-cdt)+\int_{t'}^t(C\ve s^{-3}+\frac{1}{7}s^{-2}Y(s))ds\\
&\le\frac{C\ve^2}{t^{3/2}}+\frac{|t-t'|}{6}\cdot\frac{3E\ve^2}{t}\le\frac{C\ve^2}{t^{3/2}}+\frac{4\tilde C E^2\ve^2}{t}\le\frac{5\tilde CE^2\ve^2}{t}
\end{split}
\end{equation*}
and
\begin{equation*}
\Big|\int_{t'}^t(aW_2)(s,\tilde{r}(s))ds\Big|\le\frac{C\ve^2}{t}(\frac{C\ve^2}{t^{3/2}}
+\frac{1}{t^{1/2}}\int_{\Gamma^-_{\mu}}|W_1|(dr-cdt)+\frac{C\ve}{t^{3/2}})
\le\frac{\tilde CE^2\ve^2}{t}.
\end{equation*}
Combining all these estimates above, we can achieve
\begin{equation*}
t|W_2(t,r)|\le6\tilde CE^2\ve^2.
\end{equation*}
This completes the proof of Lemma \ref{Lem6.4}.
\end{proof}

\subsection{Proof of Theorem \ref{thm2}}

\begin{proof}[Proof of Theorem \ref{thm2}]
Along the characteristics $\Gamma_{\sigma_0}^+$, we obtain the following Riccati-type ODE of $W_1$
\begin{equation*}
\frac{dW_1}{dt}(t,r(t))=L_1W_1=a_0(t)W_1^2+a_1(t)W_1+a_2(t),
\end{equation*}
where
\begin{equation*}
\begin{split}
a_0(t)&=\frac{cc'}{\sqrt{r}},\\
a_1(t)&=-\frac{cc'}{\sqrt{r}}W_2+\frac{cc'}{8r^2}u+\frac{3c'}{4r}\p_tu,\\
a_2(t)&=\frac{c'}{4r}\p_tuW_2-\frac{cc'}{8r^2}u W_2+\frac{c'}{8 r^{5/2}}u\p_tu+\frac{3c}{16 r^{5/2}}u-\frac{c}{8 r^{3/2}}\p_ru.
\end{split}
\end{equation*}
By $c(0)=1$ and \eqref{XYZ-assum}, one can find that
\begin{equation*}
|a_1|\le\frac{C\ve^2}{t^{3/2}}+\frac{C\ve}{t^2}+\frac{1}{t}\cdot\frac{3E\ve}{t^{1/2}}\le\frac{4E\ve}{t^{3/2}}
\end{equation*}
and
\begin{equation*}
|a_2|\le\frac{C\ve^3}{t^{5/2}}+\frac{C\ve}{t^3}+\frac{E\ve}{2t^2}\le\frac{E\ve}{t^2}.
\end{equation*}
Integrating these over $[\frac{1}{\ve},T_\ve]$ yields
\begin{equation*}
\begin{split}
&\int_{\frac{1}{\ve}}^{T_{\ve}}|a_1(t)|dt\le\int_{\frac{1}{\ve}}^{T_{\ve}}\frac{4E\ve}{t^{3/2}}dt\le8E\ve^{3/2},\\
&\int_{\frac{1}{\ve}}^{T_{\ve}}|a_2(t)|dt\le\int_{\frac{1}{\ve}}^{T_{\ve}}\frac{E\ve}{t^2}dt\le E\ve^2,
\end{split}
\end{equation*}
which implies
\begin{equation*}
K=\Big(\int_{\frac{1}{\ve}}^{T_{\ve}}|a_2(t)|dt\Big)\exp\Big(\int_{\frac{1}{\ve}}^{T_{\ve}}|a_1(t)|dt\Big)\le E\ve^2e^{8E\ve^{3/2}}=O(\ve^2).
\end{equation*}
Finally, we turn to compute $W(\frac{1}{\ve})$.
From \eqref{loc:energyC1}, one has
\begin{equation}\label{thm2-pf1}
|\p\p^{\le1}(u-u_{ap})|\le C_B\ve^{1.4}(1+r)^{-1/2},\quad t\le T_\ve.
\end{equation}
On the other hand, \eqref{LWE-error} leads to
\begin{equation}\label{thm2-pf2}
|\p\p^{\le1}(w_B-r^{1/2}F_0)|\le C(1+t)^{-3/2}\ln(2+t),\quad t\le T_\ve,r\ge t/2+1.
\end{equation}
According to the definition \eqref{app-solution}, we have $u_{ap}(\frac{1}{\ve})=\ve w_B(\frac{1}{\ve})$.
This, together with \eqref{thm2-pf1} and \eqref{thm2-pf2}, yields
\begin{equation*}
\begin{split}
W_1(\frac{1}{\ve})&=-\frac{1}{2c}\Big\{\frac{1}{2}r^{-1/2}\p_tu+r^{1/2}\p_{tr}^2u+\frac{1}{4}cr^{-3/2}u
-cr^{-1/2}\p_ru-cr^{1/2}\p_r^2u\Big\}\Big|_{\frac{1}{\ve}} \\
&=-\frac{r^{1/2}}{2c}(\p_{tr}^2u_{ap}-c\p_r^2u_{ap})\Big|_{\frac{1}{\ve}}+O(\ve^{1.4})\\
&=-\frac{\ve r^{1/2}}{2c}(\p_{tr}^2w_B-c\p_r^2w_B)\Big|_{\frac{1}{\ve}}+O(\ve^{1.4})\\
&=\ve F^{''}_0(r(\frac{1}{\ve})-\frac{1}{\ve})+O(\ve^{1.4}),
\end{split}
\end{equation*}
where $(\frac{1}{\ve},r(\frac{1}{\ve}))\in\Gamma^+_{\sigma_0}$.
It is easy to check that
\begin{equation*}
r(\frac{1}{\ve})=t_0+\sigma_0+\int_{t_0}^{\frac{1}{\ve}}c(\p_tu(s))ds=\sigma_0+\int_{t_0}^{\frac{1}{\ve}}(c(\p_tu(s))-1)ds+\frac{1}{\ve}.
\end{equation*}
From this and \eqref{XYZ-assum}, we arrive at
\begin{equation*}
|r(\frac{1}{\ve})-\frac{1}{\ve}-\sigma_0|\le O(\ve^{1/2}),
\end{equation*}
which leads to
\begin{equation*}
W_1(\frac{1}{\ve})=\ve F^{''}_0(\sigma_0)+O(\ve^{1.4})>K.
\end{equation*}
Applying Lemma \ref{blowup-lem} yields
\begin{equation*}
\int_{\frac{1}{\ve}}^{T_\ve}a_0(t)dt=2c'(0)(\sqrt{T_\ve}-\sqrt{\frac{1}{\ve}})+O(\ve\ln\frac{1}{\ve})
<\frac{e^{8E\ve^{3/2}}}{\ve F^{''}_0(\sigma_0)+O(\ve^{1.4})}.
\end{equation*}
Thus, one has
\begin{equation*}
\limsup_{\ve \to 0+}\ve\sqrt{T_\ve}\le\frac{1}{2c'(0)F_0''(\sigma_0)}.
\end{equation*}
This completes the proof of Theorem \ref{thm2}.
\end{proof}

\vskip 0.2 true cm

{\bf \color{blue}{Conflict of Interest Statement:}}

\vskip 0.1 true cm

{\bf The authors declare that there is no conflict of interest in relation to this article.}

\vskip 0.2 true cm
{\bf \color{blue}{Data availability statement:}}

\vskip 0.1 true cm

{\bf  Data sharing is not applicable to this article as no data sets are generated
during the current study.}

\end{document}